%% file: main_arxiv.tex
\documentclass[10pt, a4paper]{article}
\usepackage{configs/preamble}
\usepackage{configs/new_commands}

\input{shared}

\begin{document}

\maketitle

\begin{abstract}
\input{sections/00-abstract}
\end{abstract}

\sloppy

\input{sections/01-introduction}

\input{sections/02-preconditioning}

\input{sections/03-precond-multiterm}

\input{sections/04-tangent-adi}

\input{sections/05-rank-adaptive}
\input{sections/06-numerical-experiments}
\input{sections/07-conclusions}
\input{sections/08-acknowledgements}

\bibliographystyle{siam_macros/siamplain}
\bibliography{bibliography}

\end{document}

%% file: shared.tex
\title{Preconditioned Low-Rank Riemannian Optimization for Symmetric Positive Definite Linear Matrix Equations}

\author{Ivan Bioli\thanks{Department of Mathematics and Department of Civil Engineering and Architecture, University of Pavia, Italy (\href{mailto:ivan.bioli@unipv.it}{ivan.bioli@unipv.it}).} \and Daniel Kressner\thanks{Institute of Mathematics, EPFL, 1015 Lausanne, Switzerland (\href{mailto:daniel.kressner@epfl.ch}{daniel.kressner@epfl.ch})} \and Leonardo Robol\thanks{Department of Mathematics, University of Pisa, Italy (\href{mailto:leonardo.robol@unipi.it}{leonardo.robol@unipi.it}). 
}}

\ifpdf
\hypersetup{
  pdftitle={Preconditioned Low-Rank Riemannian Optimization for Symmetric Positive Definite Linear Matrix Equations},
  pdfauthor={I. Bioli, D. Kressner, and L. Robol}
}
\fi

%% file: sections/00-abstract.tex
This work is concerned with the numerical solution of large-scale symmetric positive definite matrix equations of the form
$A_1XB_1^\top + A_2XB_2^\top + \dots + A_\ell X B_\ell^\top = F$, as they arise from discretized partial differential equations and control problems. One often finds that $X$ admits good low-rank approximations, in particular when the right-hand side matrix $F$ has low rank.
For $\ell \le 2$ terms, the solution of such equations is well studied and effective low-rank solvers have been proposed, including Alternating Direction Implicit (ADI) methods for Lyapunov and Sylvester equations. 
For $\ell > 2$, several existing methods try to approach $X$ through combining a classical iterative method, such as the conjugate gradient (CG) method, with low-rank truncation. In this work, we consider a more direct approach that approximates $X$ on manifolds of fixed-rank matrices through Riemannian CG. 
One particular challenge is the incorporation of effective preconditioners into such a first-order Riemannian optimization method. We propose several novel preconditioning strategies, including a change of metric in the ambient space, preconditioning the Riemannian gradient, and a variant of ADI on the tangent space. Combined with a strategy for adapting the rank of the approximation, the resulting method is demonstrated to be competitive for a number of examples representative for typical applications.

%% file: sections/01-introduction.tex
\section{Introduction}

This paper is concerned with the numerical solution of multiterm matrix equations of the form
\begin{equation}
  \label{eq:multiterm_sylv}
  A_1XB_1^\top + A_2XB_2^\top + \dots + A_\ell X B_\ell^\top = F,
\end{equation}
where $A_i\in\Rmm$, $B_i\in\Rnn$ are known coefficient matrices, and
$F\in\Rmn$ is a known right-hand side.
This equation is equivalent to the linear system
\begin{equation} \label{eq:linearsystem}
 (B_1 \otimes A_1 + B_2 \otimes A_2 + \cdots + B_\ell \otimes A_\ell) \mathsf{vec}(X) = \mathsf{vec}(F),
\end{equation}
where $\mathsf{vec}: \mathbb R^{m \times n} \to \mathbb R^{mn}$ stacks the columns of a matrix into a long vector and 
$\otimes$ denotes the usual Kronecker product. Such matrix equations appear in the context of discretized partial differential equations (PDEs) on
tensorized domains, parametric and stochastic PDEs, bilinear and stochastic control; see~\cite{BennerBreiten2013,BennerDamm2011,KressnerTobler2011,Palitta2021,PalittaSimoncini2016,ZhangNagy2018} and the references therein.

For $\ell = 2$, the matrix equation~\eqref{eq:multiterm_sylv} becomes a (generalized) Sylvester equation and many specialized solvers have been developed for this case; see~\cite{simoncini_computational_2016} for an overview. This includes the Bartels-Stewart method~\cite{Bartels1972,Gardiner1992a}, which is suitable for dense coefficients and requires only $\mathcal O(m^3 + n^3)$ operations.
For $\ell > 2$, the development of such solvers is significantly more challenging. For example, there is no meaningful extension of the Bartels-Stewart method, unless rather strong conditions are met~\cite{Chen2020}. Hence, under \emph{no} additional assumptions on the data, to this date the only viable approach for solving~\eqref{eq:multiterm_sylv} is to apply a standard linear systems solver to~\eqref{eq:linearsystem}, which requires $\mathcal O(m^3 \cdot n^3)$ operations.

To address the case $\ell > 2$ efficiently, additional assumptions on the data need to be imposed. In particular, in the applications mentioned above it is frequently the case that the right-hand side $F$ has low-rank. Although there is only limited theoretical insight on this matter~\cite{BennerBreiten2013,Jarlebring2018,KressnerUschmajew2016}, this often implies that $X$ can be well approximated by a low-rank matrix. By directly aiming at an approximate low-rank solution to~\eqref{eq:multiterm_sylv}, without computing the exact solution first, one hopes to obtain an efficient and reasonably accurate solver. One possible approach is to apply a standard Krylov subspace method (like CG, GMRES or BiCGSTAB) to the linear system~\eqref{eq:linearsystem}, rephrase the method in terms of a matrix iteration and apply low-rank truncation to the iterates; see~\cite{BallaniGrasedyck2013,BennerBreiten2013,KressnerTobler2011} for examples. Based on (block) matrix-vector products, these methods directly benefit from sparse or low-rank coefficient matrices, allowing one to address very large-scale equations for which the full matrix $X$ could not even be stored in memory. Their non-stationary nature complicates the analysis of such truncated Krylov subspace methods; see~\cite{PalittaKuerschner2021,SimonciniHao2023} recent progress. All methods benefit from the availability of a preconditioner for which the inverse can be cheaply applied to a low-rank matrix. In particular, this is the case for preconditioners of the form $D\otimes E$; applying its inverse $D^{-1} \otimes E^{-1}$ preserves the rank.
Such Kronecker product preconditioner are often constructed by averaging the terms in~\eqref{eq:linearsystem} or solving an approximation problem; see, e.g.,~\cite{Ullmann2010}. A more elaborate choice of preconditioner takes the form
$D\otimes A + B \otimes E$, which can be constructed by, e.g., choosing the first two terms in~\eqref{eq:linearsystem}. Applying its inverse corresponds to solving a generalized Sylvester equation, which is usually executed inexactly, for example, by a few steps of the Alternating Direction Implicit (ADI) iteration; see, e.g.,~\cite{BennerBreiten2013}. For a sufficiently good preconditioner, one can also combine a fixed point iteration (instead of a Krylov subspace method) with low-rank truncation to arrive at a competitive method; see~\cite{Damm2008,Shank2016} for examples.
With stronger assumptions on the data (such as low-rank commutators); more effective solution strategies can be developed; see~\cite{BennerBreiten2013,Jarlebring2018,Richter1993} for examples.

In the following, we restrict ourselves to the symmetric positive definite (SPD) case, that is, the system matrix in~\eqref{eq:linearsystem} is SPD or, equivalently,
the linear operator
\begin{equation}
  \label{eq:operator_sylv}
  \calA: \Rmn\to\Rmn, \qquad \calA X = A_1XB_1^\top + A_2XB_2^\top + \dots + A_\ell X B_\ell^\top
\end{equation}
is SPD. In this case, it is well known that~\eqref{eq:multiterm_sylv} is equivalent to minimizing $\frac{1}{2}\inner{\calA X}{X} - \inner{X}{F}$, where $\langle \cdot, \cdot \rangle$ denotes the standard matrix inner product.
 One then obtains a low-rank approximation 
to $X$ by restricting the minimization to \[\M_r = \left\{X\in\Rmn\,:\,\rank(X) = r\right\}, \]
for some fixed choice of $r \ll m,n$. Noting that $\M_r$ is an embedded submanifold of $\Rmn$, this leads to the Riemannian optimization~\cite{boumal_introduction_2023} problem
\begin{equation}
  \label{eq:spd_optim_problem}
  \min_{X\in\M_r } f(X) \defas \frac{1}{2}\inner{\calA X}{X} - \inner{X}{F}.
\end{equation}
For Lyapunov equations ($\ell=2$), Vandereycken and Vandewalle~\cite{vandereycken_riemannian_2010} have developed
a Riemannian trust-region (RTR) approach for addressing~\eqref{eq:spd_optim_problem}. They also construct a highly efficient preconditioner for iteratively solving the second-order model in every step of RTR.
For $\ell > 2$, alternating optimization and greedy rank-one strategies have been discussed in~\cite{BreitenRingh2019,KressnerSirkovic2015}, which allow the incorporation of preconditioners indirectly through a preconditioned residual. In~\cite{SuttiVandereycken2021}, combinations of multigrid/multilevel methods with Riemannian optimizations are proposed. The direct incorporation of preconditioners into first-order Riemannian optimization methods on low-rank tensor manifolds has been discussed for higher-dimensional generalization of the Lyapunov case in~\cite{KressnerSteinlechner2016}.


The goal of this work is to develop a suitably preconditioned first-order Riemannian optimization method for solving SPD multiterm matrix equations~\eqref{eq:multiterm_sylv}. We address the optimization problem~\eqref{eq:spd_optim_problem} using the Riemannian Nonlinear Conjugate Gradient (R-NLCG) method~\cite{sato_riemannian_2022}. This approach inherently prevents the rank growth often observed in truncated CG methods \cite{KressnerSirkovic2015,KressnerTobler2011}, although it complicates the efficient incorporation of preconditioners. We interpret preconditioning as a modification of the Riemannian metric on $\M_r$, achievable through two strategies: (1) altering the inner product of the embedding space, or (2), implicitly, by preconditioning the Riemannian gradient. We show feasibility of both alternatives for simpler preconditioners of the form $D \otimes E$. Additionally, the combined use of these two preconditioning strategies allows us to extend the Lyapunov-like preconditioners developed in \cite{vandereycken_riemannian_2010,KressnerSteinlechner2016} to generalized Sylvester preconditioners of the form $D\otimes A + B \otimes E$. For the latter, we also introduce an approximate preconditioner based on a variant of the ADI method applied on the tangent space. Finally, we develop a rank-adaptive algorithm~\cite{gao_riemannian_2022, uschmajew_line-search_2014} by alternating between fixed-rank Riemannian optimization and rank updates.

The rest of this paper is organized as follows. In \cref{sec:review_riemannianopt}, we briefly review the structure of the manifold of fixed-rank matrices and the (preconditioned) R-NLCG method. \Cref{sec:preconditioners} is dedicated to discussing the efficient application of Riemannian preconditioners for SPD linear matrix equations. The rank-adaptive algorithm is introduced in \cref{sec:rlcg_rankadaptivity}. Finally, \cref{sec:numexp} presents numerical experiments comparing the performance of the proposed algorithms against existing methods.


\section{Brief overview of low-rank Riemannian optimization}
\label{sec:review_riemannianopt}
In this section, we give a brief tour on the tools needed for Riemannian optimization on $\M_r$;
see, e.g.,~\cite{boumal_introduction_2023} for more details.
These tools depend on the choice of inner product on $\Rmn$. We provide the well-known explicit expressions for the case of the standard inner product $\langle \cdot, \cdot \rangle$ in the following and extend them to a more general setting in \cref{sec:geometry_metricEXD}.

By the singular value decomposition (SVD), any matrix $X \in \M_r$ can be written as 
$X = U\Sigma V^\top$, where $U \in \mathbb R^{m\times r}$, $V \in \mathbb R^{n\times r}$ have orthonormal columns and $\Sigma \in \mathbb R^{r\times r}$ is diagonal with positive diagonal entries. The tangent space at $X$ takes the form
\begin{equation}
  \label{eq:tangent_space_factor_repr}
  \begin{split}
    \Tang_X\M_r
    = \Big\{& UMV^\top + U_pV^\top+UV_p^\top\,:\\
    & M\in\R^{r\times r}, U_p\in\R^{m\times r}, V_p\in\R^{n\times r}, U^\top U_p = 0, V^\top V_p = 0\Big\}.
  \end{split}
\end{equation}
This allows one to store $X$ efficiently in terms of its factors $(U,\Sigma,V)$ as well as a tangent vector $\xi\in\Tang_X\M_r$ in terms of the coefficients $(M,U_p,V_p)$.

\subsection{Embedded geometry of fixed-rank matrices}
\label{sec:geometry-fixed-rank}

\paragraph*{Tangent space projection}
Given an arbitrary matrix $Z\in\Rmn$, the tangent space projection
$\Proj_X$ maps $Z$ orthogonally, with respect to the choice of inner product on $\Rmn$, to $\Tang_X\M_r$.

In the standard inner product, an explicit expression for $\Proj_X(Z)$ is given by:
\begin{equation}
\label{eq:proj_formula}
\begin{split}
  &\Proj_X(Z) = Z - \Pp_U^\perp Z \Pp_V^\perp = UMV^\top + U_pV^\top+UV_p^\top, \\
  & \text{with} \quad M = U^\top Z V, \qquad U_p = ZV - UM, \qquad V_p = Z^\top U - VM^\top,
  \end{split}
\end{equation}
where $\Pp_U^\perp = I - UU^\top$ and $\Pp_V^\perp = I - VV^\top$.


\paragraph*{Transporter}
To map (transport) an element from one tangent space to another, we use orthogonal projection:
\begin{equation*}
  \T_{Y\leftarrow X} = \left. \Proj_Y\right\rvert_{\Tang_X\M_r}: \Tang_X\M_r \to \Tang_Y\M_r, \qquad \forall X,Y\in\M_r.
\end{equation*}
Based on~\eqref{eq:proj_formula}, an efficient implementation, using $\calO((m + n)r^2)$ flops, is described in \cite[Algorithm~6]{vandereycken_low-rank_2013}.

\paragraph*{Retraction} To map an element $X + \xi$ for $\xi \in \Tang_X\M_r$ back to the manifold, we make use of the metric projection retraction
\begin{equation}
  \label{eq:metric_proj_retr}
  \Retr_X(\xi) = \argmin_{Y\in\M_r} \| X+\xi-Y \|,
\end{equation}
where $\|\cdot\|$ denotes the norm induced by the inner product on $\Rmn$.

The norm induced by the standard inner product is the Frobenius norm $\frobnorm{\cdot}$, which allows one to 
compute~\eqref{eq:metric_proj_retr} by performing a truncated SVD of $X+\xi$. Using that $X+\xi$ has rank at most $2r$~\cite[\S~7.5]{boumal_introduction_2023}, this computation can be carried out in $\calO((m + n)r^2)$ operations, .

\paragraph*{Riemannian gradient and Hessian} Because $\M_r$ is embedded in $\Rmn$, the Riemannian gradient
of a smooth map $f:\M_r\to\R$  is obtained from projecting the 
Euclidean gradient of any smooth extension $\barf$: $\grad f(X) = \Proj_X(\nabla\barf(X)) \in \Tang_X\M_r$.
Similarly, the Riemannian Hessian $\Hess f(X)[\xi]: \Tang_X\M_r \to \Tang_X\M_r$ 
takes the form
\begin{equation}
  \label{eq:hess_submanifolds_curvature}
  \Hess f(X)[\xi] = \Proj_X\big(\Hess\barf(X)[\xi]\big) + \calD_\xi\big(\Proj_X^\perp(\nabla \barf(X))\big),
\end{equation}
where $\Hess\barf$ denotes the Euclidean Hessian of $\barf$ and $\calD_\xi$ is the differential of $X\mapsto\Proj_X$ at $X$ along $\xi$ \cite[Corollary~5.47]{boumal_introduction_2023}. The second term $\calD_\xi(\Proj_X^\perp(\nabla\barf(X)))$ 
is called the \emph{curvature term}, as it can be related to the curvature of the manifold~\cite[\S~6]{absil_all_2009}.
The presence of this term may render the Riemannian Hessian indefinite even when 
the Euclidean Hessian is positive definite.

\subsection{Riemannian Nonlinear Conjugate Gradient (R-NLCG)}

Given the ingredients defined above, a general line-search Riemannian optimization method~\cite{boumal_introduction_2023} takes the form
\[
  X_{k+1} = \Retr_{X_k}(\alpha_k \xi_k)
\]
for a search direction $\xi_k\in\Tang_{X_k}\M_r$ and a suitable step size $\alpha_k>0$. We specifically consider the Riemannian Nonlinear Conjugate Gradient (R-NLCG), which extends the (Euclidean) nonlinear conjugate gradient method to Riemannian manifolds. R-NLCG determines the search direction by combining the negative Riemannian gradient with the previous search direction 
$\xi_{k-1}\in\Tang_{X_{k-1}}\M$:
\begin{equation} \label{eq:rnlcgsearch}
  \xi_k = -\grad f(X_k) + \beta_k \T_{X_k\leftarrow X_{k-1}}(\xi_{k-1})
\end{equation}
for some $\beta_k\in\R$. A reasonable search direction should satisfy $\inner{\grad f(X_k)}{\xi_k}_{X_k} < 0$; we simply 
set $\xi_k = -\grad f(X_k)$ if this condition is violated by~\eqref{eq:rnlcgsearch}.
For choosing the step size $\alpha_k$, we use Armijo's backtracking procedure adapted to Riemannian optimization, as described in~\cite[Definition~4.2.2]{absil_optimization_2008}.

A comprehensive survey of methods for choosing $\beta_k$ in~\eqref{eq:rnlcgsearch} is provided in \cite{sato_riemannian_2022}. For example, $\beta_k =0$ yields the Riemannian Gradient Descent (R-GD) method.
Based on preliminary numerical experiments, we have selected the modified Hestenes-Stiefel rule \cite[\S~6.2.2]{sato_riemannian_2022}; comparable performance was observed when employing the modified Polak-Ribiere rule \cite[\S~6.2.1]{sato_riemannian_2022}, or the modified Liu-Storey rule \cite[\S~6.2.3]{sato_riemannian_2022}. Denoting $g_k \defas \grad f(X_k)$, one chooses
  $\beta_k = \max(0, \min(\beta_k^{\mathsf{HS}}, \beta_k^{\mathsf{DY}}))$
with
\begin{equation*}
  \begin{split}
    \beta_k^{\mathsf{HS}} &= \frac{\norm{g_k}^2 - \innersmall{g_{k}}{\T_{X_k\leftarrow X_{k-1}}(g_{k-1})}}{ \innersmall{g_{k}}{\T_{X_k\leftarrow X_{k-1}}(\xi_{k-1})} - \innersmall{g_{k-1}}{\xi_{k-1}}},\\
    \beta_k^{\mathsf{DY}} &=  \frac{\norm{g_k}^2}{ \innersmall{g_{k}}{\T_{X_k\leftarrow X_{k-1}}(\xi_{k-1})} - \innersmall{g_{k-1}}{\xi_{k-1}}} 
  \end{split}
\end{equation*}
when using the standard inner product.


%% file: sections/02-preconditioning.tex
\subsection{Preconditioned Riemannian Optimization}

First-order line-search methods exhibit slow convergence if the (Riemannian) Hessian at the solution
is ill-conditioned. As observed in~\cite{KressnerSteinlechner2016,vandereycken_riemannian_2010}, an ill-conditioned operator $\calA$ in~\eqref{eq:operator_sylv} can be expected to lead to such a situation.  In these cases, it is crucial to employ a preconditioner. 
In principle, for the purpose of Riemannian optimization it suffices to define the preconditioner as an SPD operator $\p_X$ on the tangent space $\Tang_X\M$.
However, in the context of matrix equations, it is most natural to search for an SPD
preconditioner $\p: \Rmn \to \Rmn$ on the ambient space and let
\begin{equation} \label{eq:generalformprecond}
 \p_X = \Proj_X \circ \p \circ \Proj_X.
\end{equation}
It can be easily verified that
$
 \inner{\xi}{\eta}_{\p_X} = \inner{\xi}{\eta}_{\p}
$
for all $\xi,\eta\in\Tang_X\M_r$, that is, $\p$ and $\p_X$ induce the same metric on $\Tang_X\M_r$.

We will consider two different ways of incorporating preconditioners into R-NLCG:

(i) Given a (simple) preconditioner $\calB$ on the ambient space, one possibility is to replace the standard inner product on $\Rmn$ by the one induced by $\calB$;
see~\cite{kasai_low-rank_2016,mishra_riemannian_2016,ngo_scaled_2012} for examples in the context of low-rank optimization. This effects the following change of Riemannian gradient:
\begin{equation} \label{eq:riemanngradientwithrespecttoB}
  \grad_{\calB} f(X) := \Proj^\calB_X \nabla_\calB \barf(X) = \Proj^\calB_X \calB^{-1} \nabla \barf(X).
\end{equation}
where $\Proj^\calB_X$ denotes the $\calB$-orthogonal projection onto $\Tang_X$.
While conceptually simple, this approach may bear practical difficulties. Unless $\calB$ has Kronecker product structure (see \cref{sec:geometry_metricEXD} below), there are no simple formulas for the tools from \cref{sec:geometry-fixed-rank}. In particular, it is difficult to carry out the $\calB$-orthogonal projection onto the tangent space efficiently for general $\calB$.

(ii) Another way to use a preconditioner $\calP$ is to replace the Riemannian gradient of $f$ with respect to the standard inner product by the one with respect to $\inner{\cdot}{\cdot}_{\p_X}$:
\begin{equation} \label{eq:precondgradient}
  \p_X^{-1} \grad f(X).
\end{equation}
Such an approach can be viewed as a quasi-Newton method when deriving $\p_X$ from an approximation of the Riemannian Hessian; see~\cite{boumal_low-rank_2015,KressnerSteinlechner2016,vandereycken_riemannian_2010} for examples. 

As we will see below, it can sometimes be beneficial to combine both approaches: Use a simple but less effective Kronecker product preconditioner $\mathcal B$ to change the metric of the ambient space and, additionally, use a more effective (and more complicated) preconditioner $\p_X$ to modify the Riemannian gradient. 
%
%
%
%
The search direction for the correspondingly preconditioned R-NLCG then takes the form
\begin{equation*}
  \xi_k = -\p_{X_k}^{-1} \grad_{\calB} f({X_k}) + \beta_k \T_{X_k\leftarrow X_{k-1}}(\xi_{k-1}),
\end{equation*}
where
$\beta_k = \max(0, \min(\beta_k^{\mathsf{HS}}, \beta_k^{\mathsf{DY}}))$
with
\begin{equation}
  \label{eq:modif_hestenes_stiefel_precond}
  \begin{split}
    \beta_k^{\mathsf{HS}} & = \frac{\inner{g_k}{\p_{X_k}^{-1}g_k}_{\calB} - \inner{g_k}{\T_{X_k\leftarrow X_{k-1}}(\p_{X_{k-1}}^{-1}g_{k-1})}_{\calB}}{\inner{g_k}{\T_{X_k\leftarrow X_{k-1}}(\xi_{k-1})}_{\calB} - \inner{g_{k-1}}{\xi_{k-1}}_{\calB}}, \\
    \beta_k^{\mathsf{DY}} & = \frac{\inner{g_k}{\p_{X_k}^{-1}g_k}_{\calB}}{\inner{g_k}{\T_{X_k\leftarrow X_{k-1}}(\xi_{k-1})}_{\calB} - \inner{g_{k-1}}{\xi_{k-1}}_{\calB}}.
  \end{split}
\end{equation}
and $g_k \defas \grad_{\calB} f(X_k)$.
Note that one needs to apply $\p_{X_k}^{-1}$ to $\grad f(X_k)$ only once per iteration.

Observe that if $X_\star$ is a non-degenerate minimizer, then the Riemannian Hessian at $X_\star$ with respect to the metric $\inner{\cdot}{\cdot}_{\p_X}$ is given by $\Hess_{\p_{X}}f(X_\star) = \Hess_{\p_{X}} (f \circ \Retr_{X_\star})(0) = \p_{X_\star}^{-1} \Hess f(X_\star)$ \cite[Proposition~5.45]{boumal_introduction_2023} when $\calB$ is identity. Therefore, when $\p_{X_\star}$ captures
dominant parts of $\Hess f(X_\star)$ then one can expect that $\p_{X_\star}^{-1} \Hess f(X_\star)$ is well-conditioned, leading to rapid local convergence of R-NLCG.

%% file: sections/03-precond-multiterm.tex
\section{Riemannian preconditioning for multiterm matrix equations}
\label{sec:preconditioners}
This section discusses different choices of (Riemannian) preconditioners for SPD multiterm matrix equations. Using~\eqref{eq:hess_submanifolds_curvature}, it follows that the Riemannian Hessian of
$f(X) = \frac{1}{2}\inner{\calA X}{X} - \inner{X}{F}$ is given by
\[
  \Hess f(X)[\xi] =  \Proj_X(\calA \xi) + \calD_\xi\big(\Proj_X^\perp(\calA X - F)\big)
\]
Ignoring the second term, which becomes negligible close to a good approximation of the solution, it appears reasonable to build the preconditioner $\calP$ (defining $\calP_X$ as in~\eqref{eq:generalformprecond})
from identifying $1$--$2$ dominant terms or combining terms of $\calA$.
Depending on the application (see the experiments in \cref{sec:numexp}), this may take the form $EXD$, $AX + XB$, or $AXD + EXB$. In the following sections, we will discuss the implementation of $\calP$ in increasing order of difficulty. Finally, in \cref{sec:ADI}, we will develop a variant of ADI that leads to a cheaper (but still effective) alternative to the exact application of preconditioners of the form 
$AX + XB$ or $AXD + EXB$.

\subsection{Preconditioned inner product with \texorpdfstring{$\calB X = EXD$}{BX = EXD}}

\label{sec:geometry_metricEXD}

We first consider preconditioning by replacing the standard inner product with the one induced by $\calB X = EXD$ for SPD matrices $D,E$. Letting $D = C_D^\top C_D$ and $E = C_E^\top C_E$ denote Cholesky decompositions, the Cholesky decomposition of the matrix representation of $\calB$ is obtained:
\begin{equation} \label{eq:kronchol}
 D\otimes E = (C_D \otimes C_E)^\top (C_D \otimes C_E).
\end{equation}
Although the change of inner product does not affect which elements are contained in the manifold $\M$ or tangent spaces, it is computationally beneficial to choose different representations for these elements. The following weighted SVD~\cite{van_loan_generalizing_1976} is an important tool for this purpose; we include its proof for the sake of illustration.

\begin{proposition} \label{proposition:weightedsvd}
  Let $E\in\Rmm$, $D\in\Rnn$ be SPD. Given $Z  \in \Rmn$ there exists a decomposition
  $Z = \tU\tSigma \tV^\top$ called \emph{weighted SVD} such that
  $\tU\in\Rmm$ is $E$-orthogonal ($\tU^\top E \tU = I$),
  $\tV\in\Rnn$ is $D$-orthogonal ($\tV^\top D \tV = I$),
  and $\tSigma \in \Rmn$ is diagonal with the 
 diagonal entries $\tilde{\sigma}_1 \ge \tilde{\sigma}_2 \ge \cdots \ge \tilde{\sigma}_{\min\{m,n\}} \ge 0$ called \emph{weighted singular values}.
\end{proposition}
\begin{proof}
  Considering the Cholesky decompositions introduced above, let $\tilde Z = U\tSigma V^\top$ be the SVD of $\tilde Z := C_E Z C_D^\top$. Then $Z = \tU\tSigma \tV^\top$, where $\tU = C_E^{-1}U$ and $\tV = C_D^{-1}V$ are $E$-orthogonal and $D$-orthogonal, respectively.
\end{proof}

For $X\in \M_r$, only the first $r$ weighted singular values are positive and we can instead consider a thin weighted SVD of the form
\begin{equation} \label{eq:thinweightedSVD}
 X = \tU\tSigma \tV^\top, \quad \tU\in\Rmr, \quad \tV\in\Rnr, \quad \tU^\top E \tU = \tV^\top D \tV = I_r, \,\tSigma \in \Rrr.
\end{equation}
In the following, we represent $X\in \M_r$ by the triple $(\tU, \tSigma, \tV)$.

Using QR decompositions $\tU = U R_U$, $\tV = V R_V$ with invertible $R_U, R_V \in \Rrr$, we can insert the substitutions
$\tilde M = R_U^{-1} M R_V^{-\top}$, $\tilde U_p = E^{-1} U_p$, $\tilde V_p = D^{-1} V_p$  into~\eqref{eq:tangent_space_factor_repr} to derive the modified tangent space representation
\begin{equation}
  \label{eq:tangent_space_factor_repr_weighted}
  \begin{split}
    \Tang_X\M_r
    = \big\{& \tU\tM\tV^\top + \tU_p\tV^\top+\tU\tV_p^\top\,:\\
    & \tM\in\R^{r\times r}, \tU_p\in\R^{m\times r}, \tV_p\in\R^{n\times r}, \tU^\top E \tU_p = 0, \tV^\top D \tV_p = 0\big\}.
  \end{split}
\end{equation}
In the following, a tangent vector $\xi\in\Tang_X\M_r$ is represented as $\xi \reprby (\tM, \tU_p, \tV_p)$.

To determine the coefficients~\eqref{eq:modif_hestenes_stiefel_precond} of R-NLCG, one needs to compute the (preconditioned) inner product of two tangent space elements $\xi \reprby (\tM, \tU_p, \tV_p)$ and $\xi' \reprby (\tM', \tU_p', \tV_p')$:
 \begin{equation} \label{eq:innerprodtangentspace}
     \begin{split}
     \inner{\xi}{\xi'}_{\calB} &= \Big\langle{E\big(\tU\tM \tV^\top + \tU_p \tV^\top + \tU\tV_p^\top\big)D}, {\tU\tM'\tV^\top +\tU_p'\tV^\top+\tU\tV_p'^\top} \Big\rangle\\
     &= \innersmall{\tM}{\tM'} + \innersmall{E\tU_p}{\tU_p'} + \innersmall{D\tV_p}{\tV_p'}.
   \end{split}
 \end{equation}

We now derive extensions of the explicit formulas discussed in \cref{sec:geometry-fixed-rank} for the inner product induced by $\calB$.

\paragraph*{Tangent space projection} Given $X\in \M_r$ in the representation $(\tU, \tSigma, \tV)$ explained above, we define for arbitrary $Z\in\Rmn$ -- in analogy to~\eqref{eq:proj_formula} -- the element
\begin{equation}
\label{eq:proj_formula_weighted}
\begin{split}
  &\xi = Z - (I - \tU\tU^\top E) Z (I-D\tV\tV^\top) = \tU\tM\tV^\top + \tU_p\tV^\top+\tU\tV_p^\top, \\
  & \text{with} \quad \tM = \tU^\top EZD \tV, \quad
    \tU_p = ZD\tV - \tU\tM, \quad
    \tV_p = Z^\top E\tU - \tV\tM^\top.
\end{split}
\end{equation}
The first expression implies that $Z-\xi$ is $\calB$-orthogonal to the tangent space $\Tang_X\M_r$ because its range is $E$-orthogonal to $\tU$ and its co-range is $D$-orthogonal to $\tV$. The second expression matches~\eqref{eq:tangent_space_factor_repr_weighted}. Noting
that $\tilde U^T E \tU_p = \tU^T E ZD\tV - \tU^T E \tU\tM = \tM - \tM = 0$
and, analogously, $\tilde V^T D \tV_p = 0$, this implies $\xi \in \Tang_X\M_r$
with the representation $\xi \reprby (\tM, \tU_p, \tV_p)$.

In summary, we have verified that~\eqref{eq:proj_formula_weighted} gives $\xi = \Proj_X^{\calB}(Z)$.

\paragraph*{Retraction} For defining a suitable retraction, we use the following straightforward extension of the Eckart--Young theorem.
\begin{proposition} \label{proposition:weightedEY}
  For $Z \in \Rmn$ let $Z = \tU\tSigma\tV^\top$ be the weighted SVD from Proposition~\ref{proposition:weightedsvd}.
  For $r \le \min\{m,n\}$, let $\tU_r, \tV_r$ denote the first $r$ columns of $\tU,\tV$ and let $\Sigma_r = \mathsf{diag}(\tilde \sigma_1,\ldots,\tilde \sigma_r)$. Then
  \begin{equation}
    \Pp_{\M_r}^{\calB}(Z) := \tU_r \tSigma_r \tV_r^\top
  \end{equation}
  solves the minimization problem $\min\{ \norm{Z-Y}_{\calB}: Y\in\M_r \}$.
  \end{proposition}
\begin{proof}
Using the Cholesky decomposition~\eqref{eq:kronchol}, it holds that $\| Z-Y\|_{\calB} = \| {\tilde Z - \tilde Y} \|_F$ with $\tilde Z = C_E Z C_D^\top$ and $\tilde Y = C_E Z C_D^\top$. Because of the equivalence between 
the weighted SVD of $Z$ and the usual SVD of $\tilde Z$ (see proof of Proposition~\ref{proposition:weightedsvd}),
the result follows from the Eckart--Young theorem.
\end{proof}

The metric projection retraction with respect to the inner product induced by $\calB$ is given by
\begin{equation} \label{eq:weightedretraction}
\Retr_X^{\p}(\xi) := \Pp_{\M_r}^{\calB}(X+\xi), \quad X \in \M_r, \quad \xi \in \Tang_X\M_r.
\end{equation}
Given representations $X = \tU\tSigma\tV^\top$ and $\xi \reprby (\tM, \tU_p, \tV_p)$, one computes 
$
  X + \xi =
  \begin{bsmallmatrix}
    \tU & \tU_p
  \end{bsmallmatrix}
  \begin{bsmallmatrix}
    \tSigma + \tM & I_r \\
    I_r           & 0
  \end{bsmallmatrix}
  \begin{bsmallmatrix}
    \tV & \tV_p
  \end{bsmallmatrix}^\top
$, confirming that this matrix has rank at most $2r$. This can be exploited to compute~\eqref{eq:weightedretraction} efficiently. Assuming $2r\leq \min( m, n )$, one first computes weighted QR decompositions
\[
\begin{bmatrix}\tU & \tU_p\end{bmatrix} = Q_UR_U, \quad 
\begin{bmatrix}\tV & \tV_p\end{bmatrix} = Q_VR_V, \quad Q_U^\top D Q_U = Q_V^\top E Q_V = I_{2r}.
\]
For simplicity, we compute these two decompositions from the standard QR decompositions of 
$C_E \begin{bmatrix}\tU & \tU_p\end{bmatrix}
$
and   $C_D \begin{bmatrix}\tV & \tV_p\end{bmatrix}$, respectively. Alternatively, 
a weighted Gram-Schmidt procedure~\cite{leon_gram-schmidt_2013} or a weighted Householder-QR decomposition~\cite{shao_householder_2023} can be used, which directly use $D,E$ and do not require the availability of the Cholesky factors $C_D,C_E$. 
Using the (standard) SVD of the $2r\times 2r$ matrix
$
  R_U \begin{bsmallmatrix}
    \tSigma + t\tM & tI_r \\
    tI_r           & 0_r
  \end{bsmallmatrix} R_V = \bar{U} \bar{\Sigma} \bar{V},
$
gives the weighted (thin) SVD
\[
  X+\xi = \left(Q_U\bar{U}\right) \bar{\Sigma} \left(Q_V\bar{V}\right)^\top.
\]
Finally, truncation in the sense of Proposition~\ref{proposition:weightedEY} yields~\eqref{eq:weightedretraction}.

Note that when computing $\Retr_X^{\p}(t \,\xi)$ for multiple values of $t$ (e.g., in a line search procedure) the weighted QR decompositions need to be computed only once.

\paragraph*{Transporter} The transporter is defined in terms of $\calB$-orthogonal tangent space projections as explained in \cref{sec:geometry-fixed-rank}.

\paragraph*{Riemannian gradient} Given the gradient $Z = \grad\bar{f}(X)$, the Riemannian gradient
is given by $\grad_{\calB} f(X) = \Proj^\calB_X \left(\calB^{-1} Z\right)$; see~\eqref{eq:riemanngradientwithrespecttoB}.
Using $\calB^{-1} Z = E^{-1} Z D^{-1}$ and the expression~\eqref{eq:proj_formula_weighted} for the tangent space projection  $\Proj^\calB_X$, we obtain that $\grad_{\calB}f(X)\in\Tang_X\M_r$ is represented by the triplet
\begin{equation}
  \label{eq:grad_repr_metricEXD}
    \tM   = \tU^\top Z \tV, \quad \tU_p  = E^{-1}(Z\tV - E\tU\tM ), \quad
    \tV_p = D^{-1}\big(Z^\top\tU - D\tV\tM^\top \big).
\end{equation}
Observe that the matrices $E\tU_p$ and $D\tV_p$, needed for determining inner products~\eqref{eq:innerprodtangentspace}, are available for free when computing $\tU_p$ and $\tV_p$. 

\subsection{Preconditioned gradient with \texorpdfstring{$\calP X = EXD$}{PX = EXD}}

\label{sec:precond_EXD}




Instead of changing the inner product, one can use the preconditioned gradient~\eqref{eq:precondgradient} to improve the convergence of R-NLCG. This requires solving a linear operator equation of the form
$\p_X(\xi) = \eta$, with $\eta = \grad f(X) \in \Tang_X\M_r$, in every iteration. For $\calP X = EXD$, this amounts to solving \begin{equation}
  \label{eq:equation_pinv_EXD}
  \Proj_X(E\xi D) = \eta \qquad \xi\in\Tang_X\M_r.
\end{equation}
Considering the parametrizatios for the known $X = U\Sigma V^\top$ and $\eta\reprby(M_\eta, U_\eta, V_\eta)$, and for the unknown $\xi\reprby(M_\xi, U_\xi, V_\xi)$, the equation~\eqref{eq:equation_pinv_EXD} is equivalent to
\begin{subequations}
  \begin{align}
    &(EU) M_\xi (V^\top DV) + EU_\xi (V^\top DV) + (EU) V_\xi^\top (DV) = U_\eta + UM_\eta, \label{eq:EXD_prec_eq1}           \\
    &(DV) M_\xi^\top (U^\top EU) + DV_\xi (U^\top EU) + (DV) U_\xi^\top (EU) = V_\eta + VM_\eta^\top, \label{eq:EXD_prec_eq2} \\
    &(U^\top E)U_\xi (V^\top D V) + (U^\top E U) V_\xi^\top (V^\top D)^\top + (U^\top E U)M_\xi (V^\top D V) = M_\eta. \label{eq:EXD_prec_eq3}
  \end{align}
\end{subequations}
From \eqref{eq:EXD_prec_eq1} and \eqref{eq:EXD_prec_eq2} one obtains
\begin{equation*}
      \begin{split}
     U\left[ M_\xi (V^\top DV) + V_\xi^\top DV \right] + U_\xi (V^\top DV) & = E^{-1}(U_\eta + UM_\eta), \\
     V\left[ M_\xi^\top (U^\top EU) + U_\xi^\top EU \right] + V_\xi (U^\top EU) & = D^{-1}(V_\eta + VM_\eta^\top).
  \end{split}
\end{equation*}
Exploiting that $U_\xi, V_\xi$ are orthogonal to $U,V$, it follows that these matrices are computed as
\[    U_\xi = \Pp_U^\perp E^{-1}(U_\eta + UM_\eta) (V^\top D V)^{-1},\quad
    V_\xi = \Pp_V^\perp D^{-1}(V_\eta + VM_\eta^\top) (U^\top E U)^{-1}. 
\]
Plugging these solutions into~\eqref{eq:EXD_prec_eq3} yields
\[
  M_\xi = (U^\top E U)^{-1}\left[M_\eta - (U^\top E)U_\xi (V^\top D V) - (U^\top E U) V_\xi^\top (V^\top D)^\top\right] (V^\top D V)^{-1}.
\]

\paragraph*{Implementation aspects and complexity} The evaluation of $(M_\xi, U_\xi, V_\xi)$ using the expressions derived above requires the solution of $r$ linear systems with the matrices $E$ and $D$. Assuming $E$ and $D$ to be sparse, this benefits from precomputing sparse Cholesky factorizations of $E$ and $D$. Additionally, $m$ linear systems with $U^\top E U$ and $n$ linear systems with $V^\top D V$ need to be solved. Using Cholesky factorizations of these two dense matrices, this requires ${\calO(r^2 (m+n))}$ operations, which matches the asymptotic complexity of carrying out a retraction or a vector transport.

{\remark{Preconditioning the inner product as described in \cref{sec:geometry_metricEXD} or preconditioning the gradient as described above lead to the same preconditioned Riemannian gradient, but not to the same R-NLCG iterations due to the different metrics used in the retraction.}}

\subsection{Preconditioned gradient with \texorpdfstring{$\calP X = AX+XB$}{PX = AX+XB}}
\label{sec:precond_AXXB}

We now consider preconditioned gradients with the Sylvester operator $\p X = AX+XB$ for SPD $A, B$. This requires applying $\p_X^{-1}$ to a tangent vector $\eta \in \Tang_X\M$, which amounts to solving the projected Sylvester equation
\begin{equation}
  \label{eq:proj_equations_AXXB}
  \Proj_X(A\xi + \xi B) = \eta \qquad \xi\in\Tang_X\M_r.
\end{equation}
This equation has been addressed in~\cite[\S~7.2]{vandereycken_riemannian_2010} for Lyapunov operators ($B=A$) and the manifold of symmetric positive semidefinite fixed-rank matrices. A more general scenario involving tensors of fixed multilinear rank was considered in \cite[\S~4.2]{KressnerSteinlechner2016}.
Paralleling the developments in~\cite{vandereycken_riemannian_2010}, we outline the solution of~\eqref{eq:proj_equations_AXXB}.

Considering again the parametrizations $\eta\reprby(M_\eta, U_\eta, V_\eta)$, $\xi\reprby(M_\xi, U_\xi, V_\xi)$, the linear operator equation~\eqref{eq:proj_equations_AXXB} is equivalent to solving the system of matrix equations
\begin{equation}
  \label{eq:original_system_AXXB}
  \begin{split}
    (\Pp_U^\perp A) U_\xi + U_\xi (V^\top B V) & = U_\eta - (\Pp_U^\perp AU)M_\xi,\\
    (\Pp_V^\perp B) V_\xi + V_\xi (U^\top A U) & = V_\eta - (\Pp_V^\perp BV)M_\xi^\top,\\
    (U^\top A U)M_\xi + M_\xi (V^\top B V)     & = M_\eta - (U^\top A)U_\xi - ((V^\top B)V_\xi)^\top.
  \end{split}
\end{equation}
Each individual matrix equation can be decoupled by performing the spectral decompositions
\begin{equation} \label{eq:spectraldecomp}
   U^\top A U = Q_A \Lambda_A Q_A^\top, \qquad V^\top B V = Q_B \Lambda_B Q_B^\top.
\end{equation}
with
$\Lambda_A = \diag(\lambda^{(A)}_1, \dots, \lambda^{(A)}_r)$,
$\Lambda_B = \diag(\lambda_1^{(B)}, \dots, \lambda^{(B)}_r)$.

After the change of coordinates $\oU_\alpha = U_\alpha Q_B$, $\oV_\alpha = V_\alpha Q_A$,
$\oM_\alpha = Q_A^\top M_\alpha Q_B$ for $\alpha \in \{\xi,\eta\}$ and $\oU = UQ_A$, $\oV = VQ_B$, 
the system~\eqref{eq:original_system_AXXB} becomes equivalent to
\begin{subequations}
  \begin{align}
    (\Pp_{\oU}^\perp A) \oU_\xi + \oU_\xi \Lambda_B & = \oU_\eta - (A\oU - \oU \Lambda_A)\oM_\xi, \label{eq:sylv_prec_eq1}                            \\
    (\Pp_{\oV}^\perp B) \oV_\xi + \oV_\xi \Lambda_A       & = \oV_\eta - (B\oV - \oV\Lambda_B)\oM_\xi^\top, \label{eq:sylv_prec_eq2}                  \\
    \Lambda_A \oM_\xi + \oM_\xi \Lambda_B                    & = \oM_\eta - (\oU^\top A)\oU_\xi - ((\oV^\top B)\oV_\xi)^\top \label{eq:sylv_prec_eq3},
  \end{align}
\end{subequations}
together with the orthogonality constraints $\oU^\top \oU_\xi = \oV^\top \oV_\xi  =  0$. Evaluating the $i$th column of~\eqref{eq:sylv_prec_eq1} yields
\[
  (\Pp_{\oU}^\perp A) \oU_\xi(:, i) + \lambda^{(B)}_i \oU_\xi(:, i) = \oU_\eta(:, i) - (A\oU - \oU \Lambda_A)\oM_\xi(:, i),
\]
Multiplying both sides of this equation with the matrix
\[
 L_u^{(i)} := \big(I_{m} + (A + \lambda^{(B)}_i I_m )^{-1}\oU \big(S_u^{(i)}\big)^{-1}\oU^\top \big)(A + \lambda^{(B)}_i I_m )^{-1}, 
\]
where $S_u^{(i)} := -\oU^\top (A + \lambda^{(B)}_i I_m )^{-1}\oU\in\R^{r\times r}$,
and exploiting $\oU^\top \oU_\xi(:,i) = 0$,
one obtains that 
\begin{equation}
  \label{eq:uxi_sol}
  \oU_\xi(:,i) = L_u^{(i)} \oU_\eta(:, i) - L_u^{(i)} (A\oU - \oU \Lambda_A) \oM_\xi(:, i).
\end{equation}
Similarly, the $i$th column of the relation~\eqref{eq:sylv_prec_eq2} gives
\begin{equation}
  \label{eq:vxi_sol}
  \oV_\xi(:,i) = L_v^{(i)} \oV_\eta(:, i) - L_v^{(i)} (B\oV - \oV\Lambda_B) \oM_\xi(i, :)^\top,
\end{equation}
with an analogous expression for $L_v^{(i)}$. Note that the formulas~\eqref{eq:uxi_sol} and~\eqref{eq:vxi_sol} still depend on the unknown $i$th column and row of $\oM_\xi$. Substituting these formulas into~\eqref{eq:sylv_prec_eq3} results in the equation
\begin{equation} \label{eq:AXXB_system_notvectorized}
  \begin{bmatrix}
    M_\xi(1, :)(\Lambda_1^{(A)} )^\top \\ \vdots \\ M_\xi(r, :)(\Lambda_r^{(A)})^\top
  \end{bmatrix}+  \big[\begin{array}{c|c|c}
      \Lambda_1^{(B)}M_\xi(:, 1) & \cdots & \Lambda_r^{(B)} M_\xi(:, r) \\
    \end{array}\big]
 = R,
\end{equation}
with the coefficient matrices
\[
 \Lambda_i^{(A)}  = \lambda_i^{(A)} I_r - (\oV^\top B) L_v^{(i)} (B\oV - \oV \Lambda_B), \quad 
 \Lambda_i^{(B)}  = \lambda_i^{(B)} I_r - (\oU^\top A) L_u^{(i)} (A\oU - \oU \Lambda_A),
\]
and the right-hand side
\begin{equation*}
  R = \oM_\eta - (\oU^\top A)W_u - ((\oV^\top B)W_v)^\top.
\end{equation*}
where the $i$th columns of $W_u$ and $W_v$ are given by $L_u^{(i)} \oU_\eta(:, i)$ and $L_v^{(i)} \oV_\eta(:, i)$, respectively. 

Clearly, the system is linear~\eqref{eq:AXXB_system_notvectorized} is linear in $M_\xi$ and can thus be reformulated as an $r^2 \times r^2$ linear system in $\mathsf{vec}(M_\xi)$. Solving this linear system determines
$M_\xi$, which can be inserted into~\eqref{eq:uxi_sol} and~\eqref{eq:vxi_sol} to
determine $\oU_\xi$ and $\oV_\xi$, respectively. Finally, reverting the change of coordinates yields $U_\xi = \oU Q_B^\top$, $V_\xi = \oV Q_A^\top$, and ${M_\xi = Q_A \oM_\xi Q_B^\top}$.

\paragraph*{Implementation aspects and complexity} The cost of computing $\xi \reprby (M_\xi, U_\xi, V_\xi)$ using the procedure described above is dominated by setting up and solving the linear system~\eqref{eq:AXXB_system_notvectorized} as well as the subsequent computation of $\oU_\xi$
and $\oV_\xi$ according to~\eqref{eq:uxi_sol} and~\eqref{eq:vxi_sol}. Let $c_A(m, r)$ denote the cost 
of solving $r$ linear systems with $A + \lambda^{(B)}_i I_m$ by, e.g., precomputing a sparse Cholesky factorization of the matrix once and performing backward/forward substitution with the triangular factor $r$ times. Then a total of $\calO(c_A(m,r)r)$ flops are needed in order to apply $(A + \lambda^{(B)}_i I_m)^{-1}$ to $\oU$, $(A\oU - \oU K)$ and $\oU_\eta(:, i)$ for all $i$. Additionally, the application of $L_u^{(i)}$ requires the solution of $\calO(r)$ linear systems with the dense $r\times r$ matrix $S_u^{(i)}$, which requires another $\calO( r^4 )$ flops in total. The other operations, like applying $A$ to $\oU$, can be expected to remain negligible in cost. Analogously, for applying $(B + \lambda^{(A)}_i I_m)^{-1}$ and $L_v^{(i)}$, a total of $\calO(c_B(n, r)r+r^4)$ flops are needed. Assuming linear complexity for the sparse direct solvers, $c_A(m,r) = \calO(mr)$ and $c_B(n,r) = \calO(nr)$, this amounts to a cost of $\calO(mr^2 + nr^2 + r^4)$ operations, which is on the level of computing a retraction or a transporter as long as $r = \calO(\sqrt{m+n})$. However, solving the $r^2 \times r^2$ linear system equivalent to~\eqref{eq:AXXB_system_notvectorized} with a direct dense method takes another $\calO(r^6)$ flops, which is feasible only for relatively small ranks. For larger ranks, an iterative method might be preferable, as it requires $\calO(r^3)$ flops per iteration to apply the operator on the left-hand side of~\eqref{eq:AXXB_system_notvectorized}.


\subsection{Preconditioning with \texorpdfstring{$\calP X = AXD+EXB$}{PX = AXD + EXB}}

\label{sec:precond_AXDEXB}
We now aim at utilizing a preconditioner of the form $\calP X = AXD+EXB$ for SPD $A, B, D, E$. When attempting to directly use this preconditioner to precondition the Riemannian gradient, it turns out that the presence of the matrices $D,E$ makes the involved linear system $\Proj_X (A\xi D + E \xi B) = \eta$ significantly more expensive to solve. In particular, the technique of \cref{sec:precond_AXXB} to decouple equations for the columns of $U_\xi$ and $V_\xi$ cannot be applied directly. To avoid this problem, we incorporate $D,E$ by adjusting the inner product, as in \cref{sec:geometry_metricEXD}, which will then allow us to resort to the Sylvester case from \cref{sec:precond_AXXB}.

Specifically, we decompose $\calP$ as follows:
\begin{equation*}
    \p = \calB \tilde{\p} \quad \text{where} \quad  \calB X = EXD, \quad  \tilde{\p} X = \calB^{-1}\p X = E^{-1}AX+XBD^{-1}.
\end{equation*}
Note that $\tilde{\p}$ remains an SPD linear operator in the inner product induced by $\calB$. Together with the identity 
$
  \inner{X}{Y}_\p = \langle {X}, {\tilde{\p}Y} \rangle_{\calB},
$
this suggests to use $\calB$ for the inner product and $\tilde{\p}$ for preconditioning the Riemannian gradient.
In other words, the Riemannian optimization tools from \cref{sec:geometry_metricEXD} are used, except that the Riemannian gradient is replaced by \[
\tilde{\p}_X^{-1} \grad_{\calB} f(X) = \tilde{\p}_X^{-1} \Proj_X^{\calB} \calB^{-1} \nabla \barf(X) \]
with $\tilde{\p}_X = \Proj_X^{\calB} \circ \tilde{\p} \circ \Proj_X^{\calB}$. Setting $\eta = \Proj_X^{\calB} \calB^{-1} \nabla \barf(X) \in \Tang_X \M_r$, this requires determining $\xi\in\Tang_X\M_r$ such that
the linear operator equation
\begin{equation}
  \label{eq:proj_equations_AXDEXB}
  \Proj_X^{\calB}(E^{-1}A\xi + \xi BD^{-1}) = \eta
\end{equation}
holds, which is structurally close to the projected Sylvester equation~\eqref{eq:proj_equations_AXXB}.

Paralleling the developments from the previous section, we now discuss a procedure for solving~\eqref{eq:proj_equations_AXDEXB} that avoids forming the matrices $E^{-1}A$ and $BD^{-1}$ explicitly.
Using the parametrizations $\eta\reprby(\tM_\eta, \tU_\eta, \tV_\eta)$ and $\xi\reprby(\tM_\xi, \tU_\xi, \tV_\xi)$ from \cref{sec:geometry_metricEXD}, the equation~\eqref{eq:proj_equations_AXDEXB} can be rearranged as
\begin{equation}
  \label{eq:original_system_AXDEXB}
  \begin{split}
    (I_m - E \tU\tU^\top)A \tU_\xi + E \tU_\xi (\tV^\top B \tV) & = E\tU_\eta - (A\tU - E\tU(\tU^\top A \tU)) \tM_\xi,\\
    (I_m - D \tV\tV^\top)B \tV_\xi + D \tV_\xi (\tU^\top A \tU) & = D\tV_\eta - (B\tV - D\tV(\tV^\top B \tV)) \tM_\xi^\top,\\
    (\tU^\top A \tU)\tM_\xi + \tM_\xi (\tV^\top B \tV)     & = \tM_\eta - (\tU^\top A)\tU_\xi - ((\tV^\top B)\tV_\xi)^\top,
  \end{split}
\end{equation}
together with the orthogonality constraints $\tU^\top E \tU_\xi =  \tV^\top D \tV_\xi = 0_r$. At this point, one can essentially follow the procedure discussed in \cref{sec:precond_AXXB}. Performing spectral decompositions~\eqref{eq:spectraldecomp} of $\tU^\top A \tU$, $\tV^\top B \tV$, and performing an analogous change of variables decouples the columns of the first two equations in~\eqref{eq:original_system_AXDEXB}. In particular, the
$i$th column of the transformed first equation reads as
\[
  (I_m - E \tU\tU^\top)A \oU_\xi(:, i) + \lambda^{(B)}_i E \oU_\xi(:, i) = E \oU_\eta(:, i) - (A\oU - E \oU \Lambda_A)\oM_\xi(:, i).
\]
Multiplying both sides of this equation with the matrix
\[
 L_u^{(i)} := \big(I_{m} + (A + \lambda^{(B)}_i E )^{-1}(E\oU) \big(S_u^{(i)}\big)^{-1}(E\oU)^\top \big)(A + \lambda_i E )^{-1}, 
\]
where $S_u^{(i)} := -(E\oU)^\top (A + \lambda^{(B)}_i E )^{-1}(E\oU)\in\R^{r\times r}$, and exploiting $\oU^\top E \oU_\xi(:, i) = 0$ again gives
\[
 \oU_\xi(:,i) = L_u^{(i)} \oU_\eta(:, i) - L_u^{(i)} (A\oU - \oU \Lambda_A) \oM_\xi(:, i).
\]
Similarly, the expression~\eqref{eq:vxi_sol} for $\oV_\xi(:,i)$ and the equation~\eqref{eq:AXXB_system_notvectorized} for $\oM_\xi$ are extended.
Note that there is no need to explicitly compute $E^{-1}A$ and $D^{-1}B$; instead, pencils of the form $A + \lambda E$ and $B + \lambda D$ are used.

\paragraph*{Implementation aspects and complexity}
The analysis of computational cost is completely analogous to the one in \cref{sec:precond_AXXB};
one only needs to replace $c_{A}(m,r)$ and $c_{B}(n,r)$ by 
the cost for solving $r$ linear systems with the (sparse) SPD matrices $A + \lambda^{(B)}_i E$
and $B + \lambda^{(A)}_i D$, respectively.

%% file: sections/04-tangent-adi.tex
\subsection{tangADI: ADI on the tangent space}
\label{sec:ADI}

If the preconditioner takes the form ${\mathcal{P} X = AXD+EXB}$ the application of $\mathcal{P}^{-1}$ entails the solution of a generalized Sylvester equation. ADI~\cite{wachspress_iterative_1988} 
is a popular strategy for iteratively solving such equations, which has been adapted to produce low-rank approximations in, e.g.,~\cite{benner_adi_2009,li_low_2002}. The goal of this 
section is to develop an ADI-like iteration on the tangent space to approximate the inverse of $\p_X = \Proj_X \circ \p \circ \Proj_X$, as a cheaper alternative to the procedures described in \cref{sec:precond_AXXB,sec:precond_AXDEXB} for the exact inversion of $\p_X$.

\subsubsection{Classical ADI}

Classical ADI~\cite{wachspress_iterative_1988} can
be derived from the observation that the Sylvester operator $\p X = AX + XB$ is a sum of two commuting linear operators, $X\mapsto AX$ and $X\mapsto XB$, and that applying (shifted) inverses of the individual summands is much simpler than applying $\p^{-1}$ as a whole. ADI extends to the solution of a generalized Sylvester equations ${AXD+EXB = F}$ by (formally) rewriting it as the standard Sylvester equation $E^{-1}A X + XBD^{-1} = E^{-1}FD^{-1}$. The resulting iteration can be rephrased such that explicit inverses of $D,E$ are avoided. Given the current iterate $X^{(j-1)} \in \Rmn$ the next iterate $X^{(j)}$ of ADI is determined by solving the matrix equation
\begin{equation}
  \label{eq:ADI_AXDEXB}
  (A-q_jE)X^{(j)}(B+p_jD) = (p_j - q_j)F  + (A-p_jE)X^{(j-1)}(B+q_jD),
\end{equation}
which amounts to multiplying both sides of the equation with $(A-q_jE)^{-1}$ and $(B+p_jD)^{-1}$
The scalars $p_j,q_j$ are called \emph{shift parameters} and their choice significantly influences convergence. Although it is common to employ different shifts at each iteration to attain fast convergence, it is worth noting that~\eqref{eq:ADI_AXDEXB} can be viewed as a fixed point iteration for constant shifts $(p_j, q_j) \equiv (p, q)$.
This fixed point iteration is derived from the operator splitting
\begin{equation}
  \label{eq:generalized_sylvester_rewriting}
  \p X = \frac{1}{p-q}\Big(\underbrace{(A - qE) X (B + pD)}_{\calG (X)} - \underbrace{(A - pE) X (B + qD)}_{\calN (X)}\Big)
\end{equation}
and converges linearly provided that $\rho\big(\calG^{-1}\calN\big)<1$, where $\rho(\cdot)$ denotes the spectral radius. Provided that the pencils $A -\lambda E$, $B -\lambda D$ are diagonalizable, the convergence rate 
is determined by their spectra $\Lambda(A, E)$, $\Lambda(B, D)$:
\[
  \rho\left(\calG^{-1}\calN\right) = \max_{\substack{\lambda\in\Lambda(A, E) \\ \mu\in\Lambda(B, D)}} \abs*{\frac{(\lambda - p)(\mu + q)}{(\lambda - q)(\mu + p)}}.
\]

\subsubsection{Basic form of tangADI}
\label{sec:tangADI}

In this section, we derive an ADI-like iteration for solving the projected generalized Sylvester equation
\begin{equation*}
  \p_X(\xi) = \Proj_X(A\xi D + E \xi B) = \eta \qquad \xi\in\Tang_X\M_r,
\end{equation*}
for given $\eta\in\Tang_X\M_r$ and SPD matrices $A, B, D, E$. Considering the decomposition $\p_X(\xi) = \Proj_X(A\xi ) + \Proj_X(\xi B)$ for $E = I_m$ and $D = I_n$, the presence of the projection $\Proj_X$ implies that the operators defining the two summands do not commute and, in turn, the usual arguments for deriving ADI do not apply.

On the other hand, we can still extend the splitting~\eqref{eq:generalized_sylvester_rewriting},
\[
  \p_X(\xi) = \frac{1}{p-q}\Big(\underbrace{\Proj_X\big( (A - qE) \xi (B + pD) \big)}_{\calG_X(\xi)} - \underbrace{\Proj_X\big( (A - pE) \xi (B + qD) \big)}_{\calN_X(\xi)}\Big),
\]
resulting in the fixed-point iteration $\xi^{(j)} = \calG^{-1}_X \big( \calN_X(\xi^{(j-1)}) + (p-q) \eta \big)$ on the tangent space $\Tang_X\M_r$.
The convergence of this iteration is determined by $\rho\big( \calG^{-1}_X \calN_X \big)$ and interlacing properties 
for eigenvalues of definite matrix pencils \cite{lancaster_variational_1991} imply
\begin{equation} \label{eq:spectralradius}
   \rho\big({\calG}_X^{-1} {\calN}_X\Big) \leq \rho \big(\calG^{-1}\calN\big),
\end{equation}
with $\calG, \calN$ defined as in~\eqref{eq:generalized_sylvester_rewriting}. In other words, the convergence rate of this fixed point iteration is not worse than the convergence rate of ADI~\eqref{eq:ADI_AXDEXB} with \emph{constant} shifts.

In practice, one observes that allowing for non-constant shifts $p_j,q_j$ benefits the convergence of the fixed point iteration on the tangent space as well, leading to the \emph{tangADI} iteration:
\begin{multline}
  \label{eq:tangADI_implicit_iteration}
  \Proj_X\big((A-q_{j}E)\xi^{(j)}(B+p_{j}D)\big) = \\ = \Proj_X\big((A-p_{j}E)\xi^{(j-1)}(B+q_{j}D)\big) + (p_{j} - q_{j})\eta.
\end{multline}

Due to the lack of commutativity mentioned above, the usual arguments~\cite{beckermann_bounds_2019} for analyzing the convergence of ADI do not apply to~\eqref{eq:tangADI_implicit_iteration} and therefore there is no rigorous theoretical basis for choosing (non-constant) shifts. Still, the inequality~\eqref{eq:spectralradius} suggests that the shifts used for classical ADI are a good choice for tangADI as well.
In our experiments, we employ the elliptic integral based (sub)optimal Wachspress' ADI shifts \cite{wachspress_adi_2013}.
Such a choice of shifts also has the advantage that it does not depend on the tangent space and, hence, the shifts can be reused across different iterations of R-NLCG.

\subsubsection{Implementation aspects and complexity}

The implementation of tangADI requires the solution of the equation~\eqref{eq:tangADI_implicit_iteration}.
For this purpose, we consider the usual parametrizations for the known quantities $X = U\Sigma V^\top$, $\eta \reprby (M_\eta, U_\eta, V_\eta)$, and for the unknown quantity $\xi^{(j)} \reprby (M_j, U_j, V_j)$.
Defining $Z_j \defas (A-p_{j}E)\xi^{(j-1)}(B+q_{j}D)$, the techniques discussed in \cref{sec:precond_EXD} yield
\begin{equation*}    
  \begin{split}
    U_{j+1} =& \Pp_U^\perp (A - q_j E)^{-1}(Z_jV + U_\eta + UM_\eta) (V^\top BV + p_jV^\top D V)^{-1} \\
    V_{j+1} =& \Pp_V^\perp (B + p_j D)^{-1}(Z_j^\top U + V_\eta + VM_\eta^\top) (U^\top A U - q_j U^\top E U)^{-1} \\
    M_{j+1} =& (U^\top AU -q_jU^\top EU)^{-1}\big[U^\top Z_jV + M_\eta - U^\top A U_{j+1} (V^\top BV + p_j V^\top DV)  \\
      &+ (U^\top AU - q_jU^\top EU) V_{j+1}^\top BV\big](V^\top BV + p_j V^\top DV)^{-1}.
  \end{split}
\end{equation*}
Note that the matrix $Z_j$ is not explicitly formed; instead the quantities $Z_jV$ and $Z_j^\top U$ appearing in these expressions are evaluated by setting $\xi^{(j)} = \left[\begin{array}{c|c} U & U_j \\ \end{array}\right]\left[\begin{array}{c|c} VM_j+V_j & V\\ \end{array}\right]^\top =: Y_jW_j^\top$ and computing
\begin{equation*}
    Z_jV = \big[(A + p_j E)Y_j\big]\big[(B + q_j D)W_j\big]^\top V, \quad 
    Z_j^\top U = \big[(B + q_j D)W_j\big]\big[(A + p_j E)Y_j\big]^\top U.
\end{equation*}

Using the formulas derived above, the two dominant costs of the $j$th iteration of tangADI are:
\begin{itemize}
  \item The solution of $r$ sparse linear systems with matrices $(A - q_j E)$ and $(B + p_j D)$, with a complexity of  $\calO( c_{(A,E)}(m,r) + c_{(B,D)}(n,r) )$.
  \item The solution of $m$ linear systems with the dense matrix $(U^\top A U - q_j U^\top E U)$ and $n$ linear systems with the dense matrix $(V^\top B V- p_j V^\top D V)$, which (using, e.g., Cholesky factorizations) requires $\calO(r^3) + \calO(r^2 (m+n)) = \calO(r^2 (m+n))$ flops. 
\end{itemize}
When $c_{(A,E)}(m,r)$, $c_{(B,D)}(n,r)$ are linear in $m,n$, we thus arrive at a total complexity of $\calO(r^2 (m+n))$.

\begin{remark}
Similar ideas to the ones in this section 
can be found in \cite{KressnerSteinlechner2016}, 
which presents a (Riemannian) truncated preconditioned Richardson iteration. In that 
work, the preconditioner is applied to the \emph{Euclidean} gradient. Instead of applying tangADI to the Riemannian gradient, this strategy requires the application of factored ADI~\cite{beckermann_bounds_2019} (fADI) to the \emph{Euclidean} gradient. 
This comes with two disadvantages:
(1) the \emph{Euclidean} gradient has significantly higher rank ($\ell\, r + r_F$ instead of $2r$) and (2) the rank of the approximation returned by fADI grows with the number of iterations. These disadvantages can be countered with low-rank truncation, which, however, comes with additional cost not needed when using tangADI. Further numerical experiments
are reported in \cite{Bioli2024}, 
and demonstrate the advantages of the method presented here.
\end{remark}

%% file: sections/05-rank-adaptive.tex
\section{Preconditioned R-NLCG with rank adaptivity}
\label{sec:rlcg_rankadaptivity}
\begin{algorithm}
  \caption{Preconditioned R-NLCG for Multiterm Linear Matrix Equations}
  \label{alg:rnlcg_multiterm}
  \begin{algorithmic}[1]
    \Require{$\calA X = \sum_{i=1}^\ell A_i XB_i^\top$ with SPD $\calA$, right-hand-side $F = F_L F_R^\top$}
    \Input{Rank $r$, initial guess $X_0 = \tU \tSigma \tV^\top \in\M_r$ (zero by default), inner product $\calB X = EXD$ with SPD $E,D$, Riemannian preconditioner $\tilde{\p}_X$, backtracking parameters $r,\tau\in(0,1)$ ($r = 10^{-4}, \tau = 0.5$ by default)}
    \smallskip
  
    \For{$k = 0, 1, 2, \dots$}
    \State Compute $\grad_{\calB} f(X_k)$ (using \eqref{eq:grad_repr_metricEXD} with $Z = \calA X_k - F$) \label{alg_line_rgrad}
    \State Compute $\tilde{\p}_{X_k}^{-1}\grad_{\calB} f(X_k)$ \Comment{Preconditioned gradient} \label{alg_line_prgrad}
    \State Compute $\beta_k = \max(0, \min(\beta_k^{\mathsf{HS}}, \beta_k^{\mathsf{DY}}))$ according to~\eqref{eq:modif_hestenes_stiefel_precond} \label{alg_line_betak}
    \State  $\xi_k = -\tilde{\p}_{X_k}^{-1}\grad_{\calB} f(X_k) + \beta_k \T_{X_k\leftarrow X_{k-1}}(\xi_{k-1})$ \label{alg_line_xik}
    \If{$\inner{\tilde{\p}_{X_k}^{-1}\grad_{\calB} f(X_k)}{\xi_k}_{\calB} \geq 0$} \Comment{If $\xi_k$ is not a descent direction} \label{alg_line_ifdescent}
    \State $\xi_k = -\tilde{\p}_{X_k}^{-1} \grad_{\calB} f(X_k)$ \Comment{Resort to R-GD} \label{alg_line_resetrgd}
    \EndIf
    \State Compute $\bar{\alpha}_k = - \frac{\innersmall{\grad_{\calB} f(X_k)}{\xi_k}_{\calB}}{\innersmall{\Proj_{X_k}^{\calB}(\calB^{-1}\calA \xi_k)}{\xi_k}_{\calB}}$ \Comment{Initial step-size for backtracking} \label{alg_line_initarmijo}
    \State Set $\alpha_k = \bar{\alpha}_k$ \label{alg_line_backtracking_1}
    \While{$f(\Pp_{\M_r}^{\calB}(X_k + \alpha_k \xi_k)) > f(X_k) + r \cdot \alpha_k \inner{\grad_{\calB} f(X_k)}{\xi_k}_{\calB}$} \label{alg_line_backtracking_2}
    \State $\alpha_k \gets \tau \cdot \alpha_k$ \label{alg_line_backtracking_3}  \Comment{Backtracking}
    \EndWhile
    \State $X_{k+1} = \Pp_{\M_r}^{\calB}(X_k + \alpha_k\xi_k)$ \Comment{Rank-$r$ truncation in $\calB$-inner product} \label{alg_line_step} 
    \EndFor
  \end{algorithmic}
\end{algorithm}

\Cref{alg:rnlcg_multiterm} summarizes our developments. It applies R-NLCG with a preconditioner $\p$ decomposed as $\tilde{\p} \calB$, where $\calB X = EXD$ with SPD $E,D$ is used to define the inner product, considering $\M_r$ as a Riemannian submanifold of $(\Rmn, \inner{\cdot}{\cdot}_{\calB})$, and 
$\tilde{\p}$ is used to precondition the Riemannian gradient through the action of $\tilde{\p}_X^{-1}$.
For $\p X = EXD$, we can choose either $\calB = \p$, $\tilde{\p} = \mathrm{id}$, or $\calB = \mathrm{id}$, $\tilde{\p} = \p$; the two choices lead to different algorithms. For $\p X = AXD+EXB$, we can set $\calB = EXD$ and $\tilde{\p} = E^{-1}AX+XBD^{-1}$, or use $\calB = \mathrm{id}$, $\tilde{\p} = \p$ and approximate the action of $\tilde{\p}_X^{-1}$ using tangADI.

At each iteration, \Cref{alg:rnlcg_multiterm} commences by computing the preconditioned Riemannian gradient (line~\ref{alg_line_rgrad}): notably, this is the only step where $\tilde{\p}_X^{-1}$ is applied. Subsequently, the R-NLCG search direction $\xi_k$ is computed, and if it is not a descent direction, we reset it to the negative preconditioned gradient (lines~\ref{alg_line_betak} to~\ref{alg_line_resetrgd}). Following \cite{vandereycken_low-rank_2013,KressnerSteinlechner2016}, the initial step size $\bar{\alpha}_k$ for Armijo backtracking is obtained by conducting an exact line search on the tangent space, neglecting the retraction (line~\ref{alg_line_initarmijo}). This initial estimate turned out to be highly effective, rarely necessitating  backtracking. Utilizing the metric projection with respect to $\mathcal{B}$ as a retraction (see~\cref{proposition:weightedEY}), starting from $\bar{\alpha}_k$, we utilize Armijo backtracking to compute the step size $\alpha_k$ (lines~\ref{alg_line_backtracking_1} to~\ref{alg_line_backtracking_3}), and finally execute the step (line~\ref{alg_line_step}).

\paragraph*{Complexity} To simplify the complexity analysis, we assume that the coefficients of $\calA, \calB,\tilde{\p}$ are sparse and that the cost of matrix-vector multiplication or solving a sparse linear system is linear with respect to the size of the sparse matrix involved.
Let $X = \tU\tSigma \tV^\top$. Writing
\begin{equation}
  \label{eq:residual_fact}
  \calA X - F= R_L R_R^\top \defas \begin{bmatrix}A_1\tU\tSigma &\cdots & A_\ell \tU\tSigma & -F_L\end{bmatrix} \begin{bmatrix}B_1\tV &\cdots & B_\ell \tV & F_R\end{bmatrix}^\top,
\end{equation}
we obtain that the cost of computing $f(X)$ and $\grad_{\calB} f(X)$ is $\calO \left(r (\ell\,r + r_F)(n+m)\right)$. Note that most of the operations needed to calculate the gradient are already performed when $f(X)$ is computed, and can therefore be reused; we refer to~\cite{Bioli2024} for details. Thus, lines~\ref{alg_line_rgrad} and~\ref{alg_line_backtracking_2} have cost $\calO \left(r (\ell\,r + r_F)(n+m)\right)$. Making simplifications similar to \cref{eq:grad_repr_metricEXD}, the cost of calculating $\bar{\alpha}_k$ is $\calO \left(r^2(n+m)\right)$. Lines~\ref{alg_line_betak} and~\ref{alg_line_step} have cost $\calO(r^2(m+n))$. Therefore, assuming $r_F = \calO(r)$, the total cost of one iteration of R-NLCG is $\calO \left(r^2(n+m)\right)$ plus the cost of applying $\tilde{\p}_{X_k}^{-1}$. In the case of a constant number of tangADI iterations, the latter has complexity $\calO \left(r^2(n+m)\right)$ as well.

\subsection{Rank adaptivity}
Until now, we  have assumed that the rank $r$ is constant and known. However, in practical applications, a good choice of $r$ is rarely known a priori, and determining it can be challenging. 
This has motivated the development of rank-adaptive techniques (see, e.g.,~\cite{gao_riemannian_2022, uschmajew_line-search_2014}), which we have extended to preconditioned R-NLCG for multiterm matrix equations.

\begin{algorithm}[ht]
  \caption{Riemannian Rank-Adaptive Method (RRAM)}
  \label{alg:rram}
  \begin{algorithmic}[1]
    \Parameters{Tolerance for truncation $\varepsilon_\sigma \in ]0,1[$, update rank $r_{\mathrm{up}}$}
    \Input{Initial guess $X_0\in\M_{r_0}$, $\calB X = EXD$ with SPD $E,D$, tolerance $\mathtt{tol} > 0$ on relative residual in $\calB$-norm}
    \smallskip
    \State Initialize $r = r_0$, $\mathrm{res} = \smallfrobnorm{\calA X_0 - F} / \smallfrobnorm{F}$, $k = 0$
    \While{$\mathrm{res} > \mathtt{tol}$}
    \While{fixed-rank optimization does not converge}
    \State One step of (preconditioned) R-NLCG, obtaining $X_{k+1} = \tU \tSigma \tV^\top \in\M_r$ \label{alg_line_rram_riemannopt}
    \If{(fixed-rank optim. did not converge) \textbf{and} ($\tilde{\sigma}_r^2 / \sum_{i=1}^r \tilde{\sigma}_i^2< \varepsilon_\sigma^2$)}
    \State $r \gets r_{-} = \max\{k\,:\,\sum_{i=k+1}^r \tilde{\sigma}_i^2 / \sum_{i=1}^r \tilde{\sigma}_i^2 \geq \varepsilon_\sigma^2 \}$ \Comment{rank decrease} \label{alg_line_rram_rank_decrease_1}
    \State $X_{k+1} \gets \tU(:, 1:r) \tSigma(1:r, 1:r) \tV(:, 1:r)^\top$ \label{alg_line_rram_rank_decrease_2}
    \EndIf
    \State $k \gets k+1$
    \EndWhile
    \State $\mathrm{res} = \smallfrobnorm{\calA X_k - F} / \smallfrobnorm{F}$ \label{alg_line_rram_res_computation}
    \If{$\mathrm{res} > \mathtt{tol}$}
    \State $X_k \gets X_k + \alpha_*Y_*$, where $Y_\star, \alpha_*$ are computed as in \eqref{eq:rank_update direction} \label{alg_line_rram_rank_increase_1}
    \State $r \gets r_+ = r + r_{\mathrm{up}}$ \label{alg_line_rram_rank_increase_2}
    \EndIf
    \EndWhile
  \end{algorithmic}
\end{algorithm}

Our Riemannian Rank-Adaptive Method (RRAM) is summarized in \Cref{alg:rram}. It alternates between fixed-rank optimization and updates that increase the rank. Beginning with an initial guess $X_0$ of rank $r=r_0$, we execute one step of R-NLCG on $\M_r$. If the current iterate becomes numerically rank-deficient (that is, the $r$th weighted singular value becomes small), the iterate is truncated to the largest rank $r_{-}$ for which the $r_{-}$th weighted singular value is not small, and the optimization restarts with rank $r_{-}$ (lines~\ref{alg_line_rram_rank_decrease_1} and~\ref{alg_line_rram_rank_decrease_2}). Upon convergence of the fixed-rank optimization, we compute the norm of the residual (line~\ref{alg_line_rram_res_computation}). If this norm is above the tolerance, we increment the rank for Riemannian optimization to $r_+ = r + r_{\mathrm{up}}$.
Following~\cite{gao_riemannian_2022, uschmajew_line-search_2014}, we construct a 
warm start for Riemannian optimization on $\M_{r_{+}}$ by adding a normal correction to the previous solution $X_k\in\M_r$ (lines~\ref{alg_line_rram_rank_increase_1} and~\ref{alg_line_rram_rank_increase_2}). For this purpose, we conduct a line search along the rank-$r_{\mathrm{up}}$ truncation of the normal component of the negative Euclidean gradient $-\grad_{\calB} \bar{f}(X_k) = \calB^{-1}(F - \mathcal{A}X_k)$. Using the same notation as in~\eqref{eq:residual_fact}, this yields the update $X_k \gets X_k + \alpha_* Y_*$, with
\begin{equation}
  \label{eq:rank_update direction}
  \begin{split}
    Y_*  &
    = \Pp_{\M_{r_\mathrm{up}}}^{\calB} \Big( \big((E^{-1} - \tU \tU^\top)R_L\big) \big((D^{-1}\tV \tV^\top)R_R\big)^\top\Big),\\
    \alpha_* &= - \frac{\innersmall{\grad_{\calB}\barf(X_k)}{Y_*}_{\calB}}{\innersmall{\calB^{-1}\calA Y_*}{Y_*}_{\calB}} =  \frac{\smallnorm{Y_*}_{\calB}^2}{\innersmall{\calA Y_*}{Y_*}}.
  \end{split}
\end{equation}
Note that the matrix $Y_*$ in~\eqref{eq:rank_update direction} is only well-defined if $\Proj_X^{\perp, \calB}(\calB^{-1}(F - \mathcal{A}X))$ has rank at least $r_{\mathrm{up}}$. If this is not the case, we add random components, $\calB$-orthogonal to the tangent space, until reaching rank $r_{\mathrm{up}}$.

We employ a heuristic strategy for halting fixed-rank Riemannian optimization by detecting a plateau in the residual norm. 
For this purpose, we compute the slope of the logarithm of an estimate of the residual norm over a backward window of \texttt{w\_len} iterations and compare it to a factor $\mathtt{fact} < 1$ times the mean slope over all previous iterations with the same rank. If the minimum slope over the last few iterations is less than this factor times the mean slope, fixed-rank optimization continues; otherwise, it halts. In our experiments, we employ Hutch++ \cite{meyer_hutch_2021} with $5$ matrix-vector multiplication to estimate the residual norm and set $\mathtt{w\_len} = 3, \mathtt{fact} = 0.75$.

\paragraph*{Complexity}
In addition to the cost of R-NLCG, the largest cost of RRAM occurs in the rank-increase steps when the residual norm and the truncated weighted SVD of $-\grad_{\calB} \bar{f}(X_k) = \calB^{-1}(F - \mathcal{A}X_k)$ are computed. Computing them using standard QR and SVD decomposition would involve a cost of $\calO\left((m+n)(\ell\, r)^2\right)$ flops, quadratic in $\ell \, r$, where $\ell$ is the number of terms in the equation. To obtain a cost linear in $\ell \, r$, one estimates the residual norm using Hutch++ \cite{meyer_hutch_2021} and could use a randomized (weighted) SVD~\cite{halko_finding_2011}. Nevertheless, for the sake of a simpler implementation, we employed a combination of QR decompositions and SVD instead of utilizing a randomized SVD.

\paragraph*{Convergence} Some convergence results for the fixed-rank variant of the algorithm have been established in \cite[\S4.3]{Bioli2024}. Local convergence to the solution $X_{\star} = \calA^{-1} F$ is ensured under the condition that $X_{\star} \in \M_r$. Global convergence can be guaranteed using a regularized cost function, employing techniques similar to those in \cite[\S4]{vandereycken_low-rank_2013}. Analyzing the convergence of the rank-adaptive version is more challenging due to the additional complexity introduced by the rank update heuristics.

%% file: sections/06-numerical-experiments.tex
\section{Numerical tests}
\label{sec:numexp}


We implemented all algorithms presented in this work in \MATLAB, R2024a. The preconditioned Riemannian methods are implemented using Manopt \cite{boumal_manopt_2013}. For R-NLCG, we used Manopt's \texttt{conjugategradient} solver, but we modified the Hestenes-Stiefel rule according to~\eqref{eq:modif_hestenes_stiefel_precond} in the preconditioned case. Unless otherwise stated, we have always used an initial guess $X_0$ that is randomly chosen to have suitable rank and Frobenius norm of $1$. 
We implemented the truncated CG method as outlined in \cite[Algorithm 2]{KressnerTobler2011}. For low-rank truncation, the following setup was found to be effective in our experiments. Given a target relative tolerance $\mathtt{tol}$ for the residual, we employed a relative truncation tolerance of $\epsilon_{\mathrm{rel}} = 0.0025 \cdot \mathtt{tol}$ for the truncation of the iterates. To prevent unnecessarily high ranks in the final CG steps, as suggested in \cite{kressner_preconditioned_2014}, we employed a mixed absolute-relative criterion for residual truncation, with $\epsilon_{\mathrm{rel}} = 0.1 \cdot \mathtt{tol}$ and $\epsilon_{\mathrm{abs}} = 0.001 \cdot \mathtt{tol}$. Using the notation from \cite[Algorithm 2]{KressnerTobler2011}, the same mixed criterion was applied to truncate $P_k$, while $Q_k$ was not truncated.  All numerical experiments were carried out on an \IntelCore{} i7-9750H 2.6 GHz CPU, featuring 6 cores, operating on a MacOS Sonoma machine with 16 GB RAM. The implementation is publicly available at \url{https://github.com/IvanBioli/riemannian-spdmatrixeq.git}.

Additional numerical experiments are reported in~\cite{Bioli2024}. In particular, we have also developed and tested a Riemannian trust-region approach, and observed it to be not competitive with R-NLCG.

In the following, we consider three problem classes representative for applications of multiterm matrix equations: PDEs on separable domains, stochastic/parametric PDEs, and control problems.

\subsection{Finite difference discretization of 2D PDEs on square domain} \label{sec:2dpde}
As a first test problem, we consider a stationary diffusion equation on a square domain with Dirichlet boundary conditions
\begin{equation}
    \label{eq:diffusion_pde}
    \begin{cases}
        -\nabla\cdot\left(k\nabla u\right)= 0 & \text{in } \Omega = [0, 1] \times [0,1],         \\
        u = g                                 & \text{on } \partial\Omega,
    \end{cases}
\end{equation}
where the diffusion coefficient $k$ is semi-separable: 
\begin{equation*}
    k(x,y) = \alpha_1 k_{1,x}(x)k_{1,y}(y) + \cdots +\alpha_{\ell_k} k_{\ell_k,x}(x)k_{\ell_k,y}(y). 
\end{equation*}
When using a standard finite difference discretization on a uniform mesh with mesh size ${h = 1/(n+1)}$ and arranging the unknowns as a matrix $U\in\mathbb{R}^{n\times n}$ such that $U_{s,k}\approx u(x_s,x_k)$ with $x_i = ih$, one obtains the matrix equation
\begin{equation}
    \label{eq:pde_fd_matrixeq}
    \sum_{j=1}^{\ell_k} \alpha_i\left(A_{j,x}^k U (D_{j,y}^k)^\top + D_{j,x}^k U (A_{j,y}^k)^\top\right) = F,
\end{equation}
where (using \MATLAB-like notation):
\begin{equation*}
    \begin{split}
        A_{j,z}^k &= \frac{1}{h^2} \mathrm{tridiag}\big(\big[-k_{j,z}(x_{i-\frac{1}{2}}),\,\,k_{j,z}(x_{i+\frac{1}{2}})+k_{j,z}(x_{i-\frac{1}{2}}),\,\, -k_{j,z}(x_{i+\frac{1}{2}})\big],\,\,{-1:1}\big), \\
        D_{j,z}^k &= \diag\left(k_{j,z}(x_1),k_{j,z}(x_2),\cdots,k_{j,z}(x_n)\right),
    \end{split}
\end{equation*}
with $z \in \{ x, y\}$. The right-hand-side matrix $F$ is given by:
\begin{equation*}
    \begin{split}
        &F = \sum_{j=1}^p e_1 b_l^\top + e_n b_r^\top + b_d e_1^\top + b_u e_n^\top, \\ 
        &[b_u]_i = k(x_i, 1-h/2) g(x_i, 1), \quad [b_d]_i = k(x_i, h/2) g(x_i,0), \\
        &[b_r]_j = k(1-h/2, x_j) g(1, x_j), \quad [b_l]_j = k(h/2, x_j) g(0, x_j). \\
    \end{split}
\end{equation*}
Assuming $k > 0$ in $\Omega$, the linear operator defined by \eqref{eq:pde_fd_matrixeq} is SPD. Moreover, all the matrices $A_{j, z}^k$ and $D_{j, z}^k$ are symmetric and, when additionally assuming $k_{j,z} > 0$, also SPD.

For our numerical experiments we consider $k(x,y) = 1 + \sum_{i = 1}^{\ell_k} \frac{\alpha^i}{i!} x^i\,y^i$, $f(x,y) = 0$, and $g(x, y) = \mathrm{exp}(-\alpha(x+1)y)$ with $\alpha = 10$, $\ell_k = 3$, and $n = 10\,000$. The resulting multiterm matrix equation~\eqref{eq:pde_fd_matrixeq} has ${\ell = 8}$ terms and $F$ has rank $r_F = 4$. 

\paragraph*{Preconditioning}
A suitable preconditioner for \eqref{eq:pde_fd_matrixeq} is derived by approximating the diffusion coefficient $k$ by  separable function $k_0(x,y) > 0$.
For the example above, we choose $k_0(x,y) = \left(1 + (\sqrt{\alpha}x)^{\ell_k} / \sqrt{\ell_k!}\right)\left(1 + (\sqrt{\alpha}y )^{\ell_k}/ \sqrt{\ell_k!}\right)$, as this form preserves both the lowest and highest degree terms in $k$. Discretizing~\eqref{eq:diffusion_pde} with $k$ replaced by $k_0$ yields a preconditioner of the form $\p^{(2)} X = A_{j,x}^{k_0} U D_{j,y}^{k_0} + D_{j,x}^{k_0} U A_{j,y}^{k_0} =: AXD+EXB$. Following~\cite[\S~4.6.2]{vandereycken_riemannian_2010-1}, a Lyapunov preconditioner can be obtained by simply dropping $E,D$, resulting in $\p^{(1)} X = AX+XB$.

\paragraph*{Numerical results}
In Figure~\ref{fig:ex1}, we compare different approaches for solving~\eqref{eq:pde_fd_matrixeq}. We used R-NLCG with rank $r=12$ employing one of the two preconditioners $\p^{(1)}$ and $\p^{(2)}$, implemented as described in \cref{sec:precond_AXXB,sec:precond_AXDEXB}, respectively. It can be seen that $\p^{(1)}$ does not lead to competitive performance relative to using $\p^{(2)}$. Using tangADI with $8$ shifts instead of $\p^{(2)}$ results in slightly slower convergence but, due to its lower cost, the execution time required to reach a small residual norm is lower. A further speedup is obtained when using rank adaptivity (RRAM), which constitutes the best choice for this example.
In RRAM the initial and update ranks are set to $r_0 = r_{\mathrm{up}} =  3$ and $\p^{(2)}$ is employed as a preconditioner.
Note that the red dots in the curves of Figure~\ref{fig:ex1} indicate rank increases of RRAM. 

We have tested CG with truncation, using fADI with the same $8$ shifts used for tangADI as a preconditioner and two different ways of low-rank truncation: (1) When choosing a low-rank truncation tolerance based on a mixed relative-absolute criterion, as described at the beginning of the section, one obtains a convergence rate similar to fixed-rank R-NLCG, confirming that its weaker theoretical foundations do not seem to impede the effectiveness of tangADI. At the same time, the ranks of the CG iterates grow quickly, leading to non-competitive time performance. (2) When capping the rank at $12$, CG is initially faster but its convergence significantly suffers from the error introduced by rank-$12$ truncations, to the extent that the method stagnates at a residual norm of $10^{-4}$, far above of what R-NLCG can attain with the same rank.

%
\begin{figure}[ht]
    \centering
    \includegraphics[width=\textwidth]{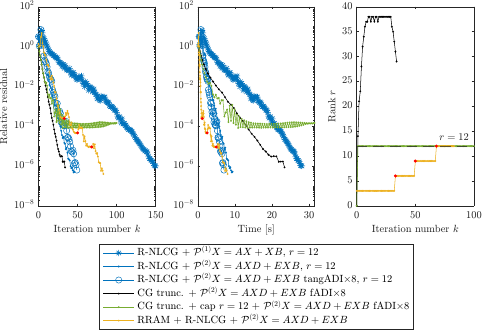}
    \caption[fig]{
    Discretized two-dimensional PDE from \cref{sec:2dpde} with $m = n = 10\,000$. Comparison of R-NLCG with fixed rank $r = 12$ and with rank adaptivity (RRAM)  as well as truncated CG with two different low-rank truncation strategies. From left to right: relative residual vs. iterations, relative residual vs. time, and rank of approximate solution vs. iterations.
    }
    \label{fig:ex1}
\end{figure}

\subsection{Stochastic Galerkin matrix equations} \label{sec:stochasticgalerkin}

We now consider a parameterized diffusion equation given by
\begin{equation}
    \label{eq:parametrized_pde_formulation}
    \begin{cases}
        -\nabla\cdot (a(x,y)\nabla u(x,y)) = f(x) & \text{in } \Omega \times \Gamma \\
        u(x,y) = 0                                & \text{on } \partial\Omega
    \end{cases}
\end{equation}
for some spatial domain $\Omega$ and the parametric domain $\Gamma = [-1,1]^q$, $q\in \mathbb N$. The parameter $a:\Omega\times\Gamma\to\R$ determining the diffusion coefficient takes the form  $a(x,y) = a_0 + \sum_{k=1}^{q}  a_k(x)\, y_k$ with $a_0 > 0$ and $ \sum_{k=1}^q \norm{a_k}_\infty < a_0$. Typically, such  
parameterized PDEs arise from a truncated Karhunen-Loève (KL) expansion of the random field in a stochastic elliptic PDE; see, e.g.,~\cite{lord_introduction_2014}.

To solve~\eqref{eq:parametrized_pde_formulation} we use stochastic Galerkin~\cite{babuska_galerkin_2004, lord_introduction_2014}, that is, we use the Galerkin method to discretize the weak formulation of~\eqref{eq:parametrized_pde_formulation} on $V^h \otimes S^p$, where $V^h$ is spanned by a finite element basis $\{\varphi_i(x)\}_{i=1}^m$ and $S^p$ contains all multivariate polynomials in $y$ of a maximal (total) degree $p$, with an $L^2$-orthonormal basis $\{\psi_j(y)\}_{j=1}^n$. This yields the multiterm matrix equation
\begin{equation}
    \label{eq:stoch_pdes_matrixeq}
    K_0X + \sum_{k=1}^{q} K_kXG_k^\top = \mathbf{f}_0\mathbf{g}_0^\top,
\end{equation}
with the matrix entries $[K_k]_{s,t}  = \int_{\Omega}  a_k\nabla\varphi_s\cdot \nabla\varphi_t$ for $k = 0,\ldots,q$
and $[G_k]_{s,t}  = \inner{y_k\psi_s}{\psi_t}$
for $k = 1,\dots, q$. The entries of the right-hand side are determined by $[\mathbf{f}_0]_{s} =\int_{\Omega} f_0\varphi_s$
and $[\mathbf{g}_0]_{s} =\inner{\psi_s}{1}$.
The locality of the finite element basis implies that the matrices $K_k$ are sparse and when choosing a Legendre basis for $S^p$, the matrices $G_k$ are also sparse.

\paragraph*{Preconditioning} The ill-conditioning of~\eqref{eq:stoch_pdes_matrixeq} is primarily caused by the stiffness matrices $K_k$ and, in turn, a simple but often effective preconditioner is obtained by simply using a constant approximation for the diffusion coefficient: $a(x,y)\approx a_0$, resulting in $\p^{(1)} X = K_0 X$. As the effectiveness of this preconditioner diminishes as the variance of $a$ increases~\cite[Theorem~3.8]{powell_block-diagonal_2009}, it can be beneficial to take more information into account. One possibility~\cite{khoromskij_tensor-structured_2011} is to average the other terms contributing to $a$, $a(x,y)\approx a_0 + \sum_{k=1}^{q}  \bar{a}_k y_k$ with $\bar{a}_j = \int_\Omega a_j(x)$, which yields the preconditioner $\p^{(2)} X = K_0 X G^\top$, with $G = I + \sum_{k=1}^{q}\frac{\bar{a}_k}{a_0} G_k$, In our preliminary experiments, this preconditioner performed very similarly to the one proposed in~\cite{Ullmann2010}. Both, $\p^{(1)}$ and $\p^{(2)}$ are incorporated into R-NLCG according to~\cref{sec:precond_EXD}.

\subsubsection{Numerical results}
We consider Examples 5.1 and 5.2 from~\cite{powell_efficient_2017}, corresponding to Test Problems 5 (TP5) and 2 (TP2) of S-IFISS \cite{bespalov_stochastic_2017}. For both problems, the spatial domain $\Omega$ is a square and $V^{h}$ contains all piecewise bilinear functions on a uniform finite element mesh with $m = 16\,129$ degrees of freedom (grid-level 7 of S-IFISS). We choose $p = 5$ for $S^p$ and obtain an orthonormal basis from tensorized Legendre polynomials in $y_1, \dots, y_q$. We include the MultiRB solver from~\cite[Algorithm~4.1]{powell_efficient_2017} in our comparison, using the implementation available at \url{https://www.dm.unibo.it/~simoncin/software.html}. For TP5 the KL expansion is truncated after $q = 9$ terms, leading to $\ell = 10$ terms in the matrix equation~\eqref{eq:stoch_pdes_matrixeq}, while for TP2 we set $q = 8$ and $\ell = 9$.

\Cref{fig:ex2_TP5} shows the results obtained for TP5, with target relative residual $\mathtt{tol} = 10^{-6}$. Due to the rapid decay of the KL expansion, it suffices to consider the simple preconditioner $\p^{(1)} X = K_0 X$. \Cref{fig:ex2_TP2} shows the results for TP2 choosing the correlation length $l = 2$ and standard deviation $\sigma = 0.3$. This turns the problem more challenging, necessitating a higher rank and the more sophisticated preconditioner $\p^{(2)} X = K_0 X G$ to achieve $\mathtt{tol} = 10^{-5}$.

For both examples, CG with truncation and fixed-rank R-NLCG exhibit similar convergence rates. Due to intermediate rank growth, CG with truncation gets more expensive in later iterations, rendering it slower than RRAM for both problems. MultiRB is significantly slower than RRAM for TP5 and slightly faster than RRAM for TP2, at the expense of a significantly larger rank. Note that for the RRAM, we employed $r_0 = 5$, $r_{\mathrm{up}} = 10$ in TP5, and $r_0 = r_{\mathrm{up}} = 30$ in TP2.


\begin{figure}[ht]
    \centering
    \includegraphics[width=\textwidth]{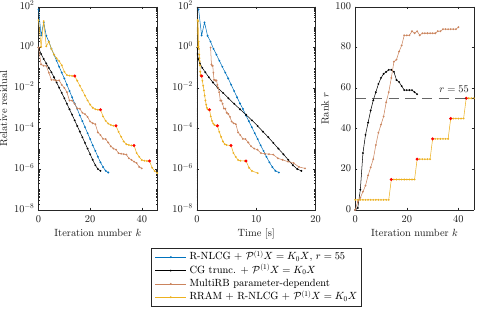}
    \caption[fig]{
        Stochastic Galerkin matrix equation from \cref{sec:stochasticgalerkin} for test problem 5 from S-IFISS ($m= 16\,129$, $n = 2\,002$, $\ell = 10$). Comparison of R-NLCG with fixed rank $r = 55$ and with rank adaptivity (RRAM) as well as truncated CG and MultiRB. 
    }
    \label{fig:ex2_TP5}
\end{figure}

\begin{figure}[ht]
    \centering
    \includegraphics[width=\textwidth]{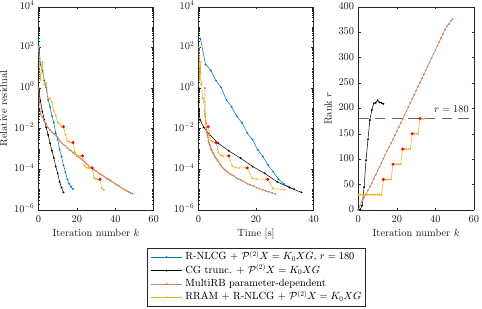}
    \caption[fig]{
        Stochastic Galerkin matrix equation from \cref{sec:stochasticgalerkin}  for test problem 2 from S-IFISS ($m= 16\,129$, $n = 1\,287$, $\ell = 9$). Comparison of R-NLCG with fixed rank $r = 180$ and with rank adaptivity (RRAM) as well as truncated CG and MultiRB. 
    }
    \label{fig:ex2_TP2}
\end{figure}

\subsection{Modified Bilinear Rail problem} \label{sec:rail}

Finally, we consider a multiterm Lyapunov equation of the form
\begin{equation} \label{eq:lyapsteel}
 (\calL-\calN) X = F, \quad \text{with}\quad \calL X= AXM+MXA, \quad \calN X = \sum_{i=1}^\ell N_i X N_i^\top,
\end{equation}
where the coefficient matrices are a modified version of those in the bilinear reformulation of the Rail example from the Oberwolfach collection \cite{benner_semi-discretized_2005}. The matrices were obtained from the M-M.E.S.S. toolbox \cite{saak_m-mess_2023}, resulting in a  multiterm linear matrix equation with size with $m = n = 5\,177$ and $\ell = 6$ terms. Adjustments were made to the constants to modify the significance of $\calN$: $\rho$ and $\gamma_k$ were divided by $10^2$, $\lambda$ and $c$ were divided by $10$, and $u_{\mathrm{ext}}$ was multiplied by $10^2$. Finally, for the right-hand side we considered $F = \tilde{B}\tilde{B}^\top$, where $\tilde{B}$ contains the first and last columns of the matrix $B$ from the M-M.E.S.S toolbox. Although these modifications may result in the loss of the original physical meaning of the equations, they yield an equation with a character that is different from the ones considered so far. In particular, the mass matrix $M$ plays a more critical role and it holds that $\rho(\calL^{-1} \calN) < 1$, which we have verified numerically. The latter implies the positive definiteness of the operator $\calL - \calN$ and that the solution inherits the symmetry and positive semidefiniteness of the right-hand-side~\cite[Lemma~5.1]{Damm2008}. This allows for performing Riemannian optimization on the manifold of fixed-rank \emph{symmetric positive semidefinite} matrices. The described Riemannian optimization tools and preconditioners described for $\calM_r$ can be easily adapted to this case; see~\cite{Bioli2024} for details.

\paragraph*{Preconditioning}
Due to the non-uniform FEM mesh used in the discretization to produce~\eqref{eq:lyapsteel}, this example features a mass matrix $M$ with a relatively high condition number ($\kappa_2(M) \approx 350$ for the chosen refinement level). As outlined in \cite[\S~8.2.3]{vandereycken_riemannian_2010-1}, this renders the Lyapunov preconditioner $\p^{(1)} X = AX+XA$, obtained from the dominant term $\calL$ by approximating $M \approx I$, less effective. We compare it with the dominant generalized Lyapunov term $\p^{(2)} X = \calL X = AXM + MXA$. 

\paragraph*{Numerical results} As an initialization strategy for Riemannian methods, we perform enough fADI steps for the generalized Lyapunov equation $AXM+M XA = \tilde{B}\tilde{B}^\top$ until we reach the desired rank. In this example, achieving a relative residual below $\mathtt{tol} = 10^{-6}$ requires a relatively high rank of $r = 150$ compared to the matrix size $n \approx 5000$. Consequently, the computational advantages of tangADI become evident. While inverting the Riemannian preconditioner exactly involves solving $4r^2+r$ linear systems of size $n$, each tangADI iteration requires solving only $r$ linear systems. Given the high rank $r = 150$, one iteration of R-NLCG with a few tangADI steps is much more efficient than using the exact inverse of the preconditioner.

The obtained results are shown in~\cref{fig:ex3}. As expected, using the preconditioner $\p^{(1)}$ instead of $\p^{(2)}$ leads to significantly higher iteration counts and longer execution times for R-NLCG. Once again, the effectiveness of tangADI is confirmed; using $8$ tangADI steps as a preconditioner does not significantly impact the R-NLCG iteration counts compared to using the exact inverse of the  preconditioner. Note that one iteration with the exact preconditioner's inverse is about $4$ times slower than the entire convergence time with tangADI.

Preconditioned CG with truncation converges in only 3 iterations without excessive rank growth, making it the best method in terms of iterations. Despite this, fixed-rank R-NLCG and RRAM with tangADI preconditioner show comparable performance in time, albeit requiring more iterations. Remarkably, even when CG with truncation performs excellently, Riemannian methods exhibit comparable time performance.

\begin{figure}[h]
    \centering
    \includegraphics[width=\textwidth]{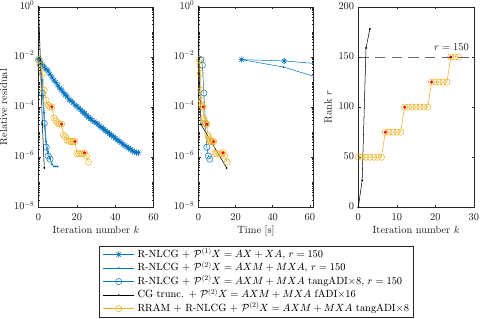}
    \caption[fig]{
        Modified Bilinear Rail problem ($n = 5\,177$) from \cref{sec:rail}.
        Comparison of R-NLCG on manifold of low-rank symmetric positive semidefinite matrices with fixed rank $r = 150$ and with rank adaptivity (RRAM) as well as truncated CG.
    }
    \label{fig:ex3}
\end{figure}

%% file: sections/07-conclusions.tex
\section{Conclusions}
First-order Riemannian optimization methods lead to relatively simple low-rank solvers for multiterm matrix equations. However, preconditioning is a challenge: existing preconditioners are defined on the ambient space, which makes it difficult to incorporate them into Riemannian optimization. In this work, we have addressed this challenge with several novel preconditioning strategies. Among them, tangADI is particularly promising. Together with rank adaptivity, this leads to a new solver that is competitive with a popular iterate-and-truncate approach.

%% file: sections/08-acknowledgements.tex
\section*{Acknowledgments}
The work of Ivan Bioli in this manuscript was carried out during his time at EPFL and the University of Pisa. 

Leonardo Robol is a member of the INdAM research group GNCS, and acknowledges support from the National Research Center in High Performance Computing, Big Data and Quantum Computing (CN1 -- Spoke 6), from the MIUR Excellence Department Project awarded to the Department of Mathematics, University of Pisa, CUP I57G22000700001, 
and from the PRIN 2022 Project ``Low-rank Structures and Numerical Methods in Matrix and Tensor Computations and their Application''.